\documentclass[11pt,a4paper,reqno]{amsart}
\usepackage{amsmath,amscd,mathtools}
\usepackage{amssymb}
\usepackage{amsthm}
\usepackage{enumerate}
\usepackage[abbrev,nobysame,alphabetic]{amsrefs}
\mathtoolsset{showonlyrefs}
\usepackage[colorlinks=true,citecolor=blue]{hyperref}
\hypersetup{urlcolor=blue, citecolor=blue, linkcolor=blue}
\usepackage{graphicx}
\usepackage{mathrsfs}
\usepackage{mathtools}
\usepackage[fleqn]{nccmath}
\usepackage{cancel}

\usepackage{calc}
{\begin{list}{\arabic{enumi}.}{\usecounter{enumi}%
			\setlength{\labelsep}{0.5em}%
			\settowidth{\labelwidth}{(\arabic{enumi})}%
			\setlength{\leftmargin}{\labelwidth+\labelsep}}}%
	{\end{list}}

\usepackage{comment}
\usepackage{xcolor}
\DeclareRobustCommand{\SkipTocEntry}[5]{}
\newcommand{\gc}{i_{g_{\frac{1}{c}}}}
\newcommand{\qum}[1]{\quad\mbox{#1}}
\newtheorem{theorem}{Theorem}
\newtheorem{definition}{Definition}[section]

\newtheorem{lemma}[definition]{Lemma}
\newtheorem{proposition}[definition]
{Proposition}
\newtheorem{remark}[definition]{Remark}
\renewcommand{\i}{\mathrm{i}}
\DeclareMathOperator{\diag}{diag}
\newcommand{\Rn}{\mathbb{R}^n}
\newcommand{\R}{\mathbb{R}}
\renewcommand{\S}{\mathbb{S}}
\newcommand{\D}{\mathrm{d}}

\newcommand{\F}{\mathcal{F}}
\newcommand{\Bdl}{B_{n}(\omega_0,\delta)}
\newcommand{\om}{\omega}

\newcommand{\Nt}{\mathbb{N}}
\newcommand{\sn}{\mathbb{S}^{n-1}}

\newcommand{\PD}{\partial}

\newcommand{\w}{\mathrm{w}}
\newcommand{\Rno}{\mathbb{R}^{1+n}}

\title[The Momentum Light Ray Transform
]{The Momentum Light Ray Transform}
\author[Bhattacharyya]{Sombuddha Bhattacharyya $^\ast$}
\address {$^{\ast}$ Department of Mathematics, Indian Institute of Science Education and Research, Bhopal.
	\newline
	E-mail:{\tt\  sombuddha@iiserb.ac.in}}
\author[Mondal]{Tuhin Mondal$^\ast$}
\address {$^{\ast}$ Department of Mathematics, Indian Institute of Science Education and Research, Bhopal.
	\newline
	E-mail:{\tt\  tuhin22@iiserb.ac.in}}

\author[Sahoo]{Suman Kumar Sahoo $^\dagger$ }
\address{ $^\dagger$ Department of Mathematics, Indian Institute of Technology Bombay.}
\email{suman@math.iitb.ac.in}
\date{}
\begin{document}

\begin{abstract}
In this article we study \emph{Momentum Light Ray Transform (MLRT)} on symmetric tensor fields. MLRT is an integral transform in time-space domain ($(t,x)\in \mathbb{R}^{1+n}$), which integrates a scalar function or a tensor field along the light rays with a polynomial type weight. 
We explore necessary and sufficient conditions for injectivity of MLRT, over general order tensors on space dimension $\geq 2$, from full and restricted measurements.
Furthermore, we develop an inversion algorithm for MLRTs in the restricted measurement setting. 
To prove the results, we use tools from tensor tomography, geometry, and analysis.
\end{abstract}


\maketitle

\section{Introduction and main results}
The study of integral geometric inverse problems dates back to the seminal work of J. Radon \cite{Radon}, who introduced the \emph{Radon transform} that integrates a function over hyperplanes in $\mathbb{R}^n$. In two dimensions, the Radon transform is equivalent to the X-ray transform, which integrates along straight lines. These transforms have broad applications in fields like medical imaging and geophysics. 
The central questions of kernel description, inversion, and range characterization are well-studied for these classical transforms \cite{Sharafutdinov_book, Paternain_Salo_Uhlmann_2023}. In contrast, the \emph{Light Ray Transform (LRT)}, which integrates a function along light rays in $\mathbb{R}^{1+n}$, has received less attention. 
The light rays (unit speed) are lines parallel to the direction $(1, \omega)$ for a unit vector $\omega \in \mathbb{R}^n$.  This transform was first introduced in the works \cite{stefanov1989inverse,ramm1991inverse} to the best of our knowledge. 
In this article, we study \emph{Momentum Light Ray Transform (MLRT)} acting on symmetric tensor fields, which is a generalization of the LRT by taking higher order symmetric tensors.
MLRT as well generalizes the \emph{Momentum Ray Transform (MRT)} by restricting directions of integrations. To this end, we begin by defining the necessary notation to introduce these transforms.      

Let $n\geq 2$ and $(t=x_0, x_1, \cdots, x_n)=(t,x)\in \R\times\Rn$. Let $S^m$ be the complex vector space of symmetric tensor fields of rank $m$. Let $C_{c}^{\infty}(\mathbb{R}^{1+n};S^m)$ be the space of all compactly supported, smooth symmetric $m$-tensor fields on $\mathbb{R}^{1+n}$ and $\mathcal{S}(\Rno; S^m)$ be the Schwartz space of symmetric $m$ tensor fields in $\Rno$.  For any $f^{(m)} \in C_{c}^{\infty}(\mathbb{R}^{1+n};S^m)$, the components of $f^{(m)}$ are defined as 
	\[
	f^{(m)}_{i_1\dots i_m}(t,x)=f^{(m)}(t,x)(e_{i_1},\dots,e_{i_m}), \quad m\ge 1,
	\]
	where $\{e_i:i=0,1,\dots,n\}$ is the standard Euclidean basis vectors of $\R^{1+n}$.  When $m=0$, we consider $f^{(0)}$ to be a scalar  function.

    \begin{definition}\cite{ramm1991inverse}\label{def_lrt_1}
		The \emph{Light Ray Transform (LRT)} of $f^{(m)}$ is a linear mapping 
		\[ \tilde{L}^m: C_{c}^{\infty}(\mathbb{R}^{1+n};S^m) \longrightarrow C^{\infty}(\mathbb{R}^{1+n}\times\sn) \]
		defined for each $(t,x)\in \Rno, \omega \in \sn$ as
		\begin{equation}\label{def_lrt}
			\tilde{L}^{m}f^{(m)}((t,x),\omega)\coloneqq \sum_{i_{1},\dots,i_{m}=0}^{n}\Tilde{\omega}_{i_{1}}\dots \Tilde{\omega}_{i_{m}}\int_{\mathbb{R}}f_{i_{1}\dots i_{m}}^{(m)}((t,x)+s\Tilde{\omega})\D s , \quad \mbox{where $\Tilde{\omega}=(1,\omega)$. }
		\end{equation}
	\end{definition}
	One can extend  $\tilde{L}^m$, for all $\om \in \R^{n}\setminus\{0\}$, denoted by $L^m$ in the following way 
	\begin{equation}\label{def_extended_lrt}
		L^{m}f^{(m)}((t,x),\omega)\coloneqq \sum_{i_{1},\dots,i_{m}=0}^{n}\Tilde{\omega}_{i_{1}}\dots \Tilde{\omega}_{i_{m}}\int_{\mathbb{R}}f_{i_{1}\dots i_{m}}^{(m)}((t,x)+s\Tilde{\omega})\D s , \quad\mbox{where $\omega\in \Rn\setminus\{0\}$}.
	\end{equation}
	We now define the momentum light ray transform, which is the main objective of this paper.
	\begin{definition}
	For each non-negative integer $k$, we define the $k$-th \emph{Momentum Light Ray Transform (MLRT)} as a linear map
		\[
		L^{m,k}: C_{c}^{\infty}(\mathbb{R}^{1+n};S^m) \longrightarrow C^{\infty}(\mathbb{R}^{1+n}\times(\R^{n}\setminus\{0\}))
		\]
		for a symmetric $m$-tensor field $f^{(m)} \in L^1(\R^{1+n};S^m)$ as
		\begin{equation}\label{def_mlrt}
			L^{m,k}f^{(m)}((t,x),\om) := \sum_{i_{1},\dots,i_{m}=0}^{n}\Tilde{\omega}_{i_{1}}\dots \Tilde{\omega}_{i_{m}}\int_{\mathbb{R}}s^kf_{i_{1}\dots i_{m}}^{(m)}((t,x)+s\Tilde{\omega})ds ,
		\end{equation}
		where $\Tilde{\omega}=(c,\omega)$, $c>0$; $(t,x)\in \mathbb{R}^{1+n}$, $\omega \in \mathbb{R}^{n}\backslash\{0\}$. 
	\end{definition}
Following the same notation, we define $\Tilde{L}^{m,k}: C_{c}^{\infty}(\mathbb{R}^{1+n};S^m) \longrightarrow C^{\infty}(\mathbb{R}^{1+n}\times\sn),$ as the restriction of $L^{m,k}$ over $\R^{1+n} \times \sn$, i.e. 
	 $\tilde{L}^{m,k}f^m= L^{m,k} f^{m}|_{\Rno\times \sn}$ for all $f^m \in   C_{c}^{\infty}(\mathbb{R}^{1+n};S^m)$.
	Note that, for $k=0$ and $c=1$, $\tilde{L}^{m,k}$ coincides with $\tilde{L}^{m}$ in Definition \ref{def_lrt_1}.
	For notational simplicity, we do not write the dependency of $L^{m,k}$ and $\tilde{L}^{m,k}$ on $c>0$.
    
Our first injectivity result for MLRT is as follows.
\begin{theorem}\label{thm:whole_inj_MLRT2} Assume that, $f^{(m)} \in C_{c}^{\infty}(\mathbb{R}^{1+n};S^{m})$ for $m\geq 0$ and $n\ge 2$. If $L^{m,m}f^{(m)}((t,x),\om)=0$, for all $(t,x)\in \R^{1+n}$, $\om \in \R^{n}\backslash \{0\}$, then $f^{(m)}=0$.
	\end{theorem}
	Notice that this injectivity result closely parallels the corresponding one for the MRT; see \cite{krishnan2018momentum}. In particular, the result shows that having no restrictions on the direction vectors, the highest-order MLRT suffices to guarantee the uniqueness of symmetric tensor fields. Allowing $\omega$ to range  over $\Rn \setminus \{0\}$, it enables us to differentiate  along all possible directions, and hence obtain a global uniqueness result. However, if we restrict $\omega$ to $\sn$ or to an open subset of $\sn$, global uniqueness no longer holds true. One of the main reasons is that, unlike in the case of MRTs, the transforms $L^{m,k}$ and $\tilde{L}^{m,k}$ are not equivalent.
	With this distinction in place, we are now ready to state our next main result, after recalling a few notations.
	 \begin{enumerate}[i)]
	 	\item For $c > 0$, let $g_c$ be the Minkowski metric with $\diag(-c^2, 1, \dots, 1)$ and zero elsewhere.
	 	\item For any $v \in S^2$, we define the operator $i_v : S^{m-2}(\R^{1+n}) \to S^m(\R^{1+n})$ by
	 	\begin{equation}
	 		i_v u := (v \odot u)_{i_1\dots i_{m}}
	 		=\frac{1}{m!}\sum_{\pi \in \Pi_{m}} u_{i_{\pi(1)}\dots i_{\pi(m-2)}}v_{i_{\pi(m-1)}i_{\pi(m)}},
	 	\end{equation}
        where $\Pi_m$ is the permutation group of $m$ elements and $\odot$ denotes the symmetric tensor product of tensors.
	 	\item For any $v \in S^2(\R^{1+n})$, we define the contraction operator $j_{v}$ as
	 	\[
	       j_{v}:S^m \to S^{m-2},\qquad 
	 	(j_{v}u)_{i_1\dots i_{m-2}}  = \sum_{i_{m-1},i_{m}=0}^{n}u_{i_1 \dots i_{m}}v_{i_{m-1}i_{m}}.
	 	\]
	 \end{enumerate}
	 	\begin{theorem}\label{thm_fullsph}
	 	Let $n\ge 2$, $m\geq 0$, and  $f^{(m)}\in C_{c}^{\infty}(\mathbb{R}^{1+n};S^m)$. Suppose that $	\Tilde{L}^{m,m}f^{(m)}((t,x),\omega)=0 \quad \forall \omega \in \S^{n-1}, (t,x) \in \mathbb{R}^{1+n},$
	 	then  \begin{align*}
	 		f^{(m)}&=\begin{cases}
	 			i_{g_{\frac{1}{c}}}f^{(m-2)}  \qum{for some} \quad f^{(m-2)} \in C_{c}^{\infty}(\mathbb{R}^{1+n};S^{m-2}) \qum{when $m\ge 2$,} \\
	 			0  \hspace{1.4cm}\qum{when}\quad m=0,1.
	 	\end{cases}
	 		\end{align*}
	 \end{theorem}
	We next restrict $\tilde{L}^{m,m} f^m$ to an open set subset   of $\sn$  denoted by $	B_{n}(\omega_0,\epsilon)$, and defined by 
	\begin{equation}\label{open_set_sphere}
		B_{n}(\omega_0,\epsilon) := \{ \omega \in \S^{n-1} : |\omega - \omega_0| < \epsilon\},
		\quad \mbox{for a fixed } \omega_0 \in \S^{n-1} \,\, \mbox{and $ \epsilon>0$},
	\end{equation}
	With this notation at hand, we now state the third main result which is 
	\begin{theorem}\label{thm_partsph}
		Let $n\geq 3$ and $f^{(m)}\in C_{c}^{\infty}(\mathbb{R}^{1+n};S^m)$, where $m\geq 0$. Fix $\om_0 \in\sn$ and $\delta>0$.  
		Then $\Tilde{L}^{m,m}f^{(m)}((t,x),\omega)=0$ for all $ \omega \in B_{n}(\omega_0,\delta)$, and for all $(t,x) \in \mathbb{R}^{1+n}$ implies 
		\begin{align*}
			f^{(m)}&=\begin{cases}
				i_{g_{\frac{1}{c}}}f^{(m-2)}  \qum{for some} \quad f^{(m-2)} \in C_{c}^{\infty}(\mathbb{R}^{1+n};S^{m-2}) \qum{when $m\ge 2$,} \\
				0 \qum{when}\quad m=0,1.
			\end{cases}
		\end{align*}
	\end{theorem}
	\begin{remark}
	Note that Theorem \ref{thm_partsph} does not hold in two dimensions, since there are not enough available directions $\omega$. However, the above result can still be established in two dimensions if one allows $\omega \in B_n(\pm \omega_0, \delta)$ for some $\delta>0$, with a slight modification of the proof of Theorem \ref{thm_fullsph}.  A counterexample, and further discussion can be found in \cite[Theorem 2.1]{siamak_vector_field_lrt}.
\end{remark}
	Now, we turn our attention to the reconstruction of tensor fields from MLRT. We present an algorithm for reconstructing a symmetric tensor field from its MLRTs. To simplify the calculations, we consider $c=1$ for the reconstruction theorems. For a fixed $\delta>0$ and a fixed $\om_0 \in \sn$, we define the set
	\[
	H_n := \{\zeta=(\zeta_0,\zeta') \in \R^{1+n}: |\zeta_0|<|\zeta'|, \quad \zeta \perp (1,\om), \quad \text{for some }\om \in \Bdl \}. \]
	The next result specifically concerns the reconstruction algorithm for a vector field.
	\begin{theorem}\label{rec_roof_fn}
		Let $n\geq 3$ and assume that for $f^{(1)}$ $\in C_{c}^{\infty}(\mathbb{R}^{1+n};S^{1})$, we know $\tilde{L}^{1,k}f^{(1)}((t,x),\omega )$ for $(t,x)\in \mathbb{R}^{1+n}$, $\omega \in B_{n}(\omega_0,\delta)$ and $k=0,1$. Then $\widehat{f^{(1)}}$ can be recovered in the open set $H_n$.\\
		More specifically, for any $\zeta \in H_n$ we have
		\[
		\hat{f}(\zeta)
		=\frac{1}{2}[\Phi_2-\Phi_4]\zeta - M^{-1}\Phi_3,
		\]
		where $\Phi_j$(see Section \ref{sec6}, equations \eqref{eq_1}, \eqref{eq_2}, and \eqref{eq_3} for more details) for $j=1,\cdots,4$ are fully determined by $\tilde{L}^{1,k}f^{(1)}((t,x),\omega)$, $k=0,1$. 
	\end{theorem}
	\begin{remark}\label{reconstruction_analytic-cont.}
	Note that $H_n$ is an open set containing space-like vectors in $\R^{1+n}$. Thus, the above reconstruction theorem is valid for an open subset of the space-like cone in $\R^{1+n}$. Given that $f^{(1)}$ is compactly supported, using Paley-Wiener theorem, one can reconstruct $\widehat{f^{(1)}}$ in the missing portions.
	\end{remark}

We have established the reconstruction algorithm for vector fields and now extend this framework to symmetric tensor fields of arbitrary rank.
	\begin{theorem}\label{rec_mt}
		Let $n\geq 3$ and $m\geq 2$. Assume that, $f^{(m)}  \in C_{c}^{\infty}(\mathbb{R}^{1+n};S^m)$ satisfies $j_{g}f^{(m)} = 0$. If $\tilde{L}^{m,k}f^{(m)}((t,x),\omega)$ is known for $k=0,1,\dots ,m$; $(t,x)\in \mathbb{R}^{1+n}$ and $\omega \in B_{n}(\omega_0,\delta)$ for some fixed $\omega_0 \in \mathbb{S}^{n-1}$, $\delta >0$, then we can recover $f^{(m)}$.
	\end{theorem}

The light ray transform arises naturally in the study of wave propagation. Solutions of the wave equation, in homogeneous medium, carry singularities that travel at the speed of light along straight paths, often referred to as light rays. To analyze or reconstruct information about the underlying medium, one is therefore led to consider integrals of functions or tensor fields along these trajectories, which is precisely the definition of the light ray transform. This approach was first introduced in the works of Stefanov, Ramm, and Sj\"ostrand \cite{stefanov1989inverse, ramm1991inverse}, where injectivity of LRT for functions was established.
Subsequently, the works \cite{krishnan2020inverse,krishnan2020uniqueness} extended the LRT to vector fields and symmetric two tensor fields, establishing corresponding injectivity results. The work \cite{feizmohammadi2021light} studied the LRT for symmetric $m$ tensor fields  in stationary and static Lorentzian geometries, reducing tensor tomography via LRT to the spatial X-ray transform. A support theorem for the LRT acting on functions was proved in \cite{stefanov2017support}, while \cite{siamak_vector_field_lrt} established a support theorem for the LRT on vector fields.  A microlocal analysis of the LRT, viewed as a Fourier integral operator, was developed in \cite{LLSU_lrt_Lorentzian}, which established recovery of spacelike singularities in the absence of conjugate points. The use of microlocal analysis within the framework of integral geometry dates back to \cite{GU_duke,Greenleaf_Uhlmann_duke}. For further results on parametric constructions and their connections to inverse problems in wave equations and cosmology, see \cite{Yiran_wang_lrt,Yiran_integral_geo_wave_equ,Yiran_cosmology_tom}.
More recently, the study of the LRT has been extended to pseudo-Riemannian settings in \cite{Ilmavirta_JDG, agrawal2025light}, highlighting its versatility and deep connections with both geometry and inverse problems.
	
The momentum ray transform (MRT) was introduced in the works \cites{Sharafutdinov_1986_momentum,Sharafutdinov_book}, where the kernel of the MRT acting on compactly supported symmetric tensor field distributions was described. A support theorem result for MRT was established in the work \cite{Anuj_Mishra_support} on analytic manifolds. Uniqueness, explicit inversion formulas, and range characterizations of the MRT for Schwartz-class symmetric tensor fields were obtained in \cite{krishnan2018momentum,krishnan2020momentum}, while analogous results for partial MRT were studied in \cite{MS_Siam,Mishra_Sahoo_PAMS}. More recently, unique continuation properties (UCP) for the MRT were established in \cite{AKS_UCP,IKS_Ucp}. In addition, \cite{IKS_Ucp} introduced a fractional version of the MRT for functions and established a UCP in that setting, which was subsequently extended to symmetric tensor fields in \cite{KJKP_fractional_mrt}.
A normal operator approach to study MRT was initiated in \cite{AKS_UCP}, and this perspective was subsequently employed in \cite{JKKS_I,JKKS_II} to derive inversion formulas in terms of the MRT’s normal operator. 
	
The study of ray transforms has been strongly motivated by applications in medical imaging, geophysics, and other branches of science. Moreover, ray transforms have emerged as essential tools in addressing Calder\'on-type problems for second-order PDEs, such as the Schr\"odinger and Maxwell equations, see \cite{Syl_Uhl,DKSU_invention,KSU_Maxwell,Uhlmann_survey}. More recently, MRTs have also been utilized to study inverse problems related to  biharmonic and polyharmonic operators, see \cite{BKS_MRT_polyharmonic,SS_linearized,BK_polyharmonic_localdata,BKSU_biharmonic_nonlinear}. We refer the survey \cite{Krishnan_Sahoo_IPMS} for more results related to MRT. This naturally suggests that the momentum light ray transforms (MLRTs) may provide a framework for solving inverse problems associated with higher-order wave equations of the form $\square^2 + \text{(lower order terms)}$, see  \cite{Isakov_biwave_type,BK_biwave}, just as the LRT has played a central role in studying inverse problems for wave equations due to singularity propagation along light rays (null bicharacteristics). 
A general result in this direction appears in \cite{Plamen_Yang,OSSU_real}, where the authors investigated inverse problems for real principal type operators (such as the wave and Tricomi operators) and introduced a null bicharacteristics transform for the unknowns and the latter transform was further studied in \cite{Marco_Salo_Tzou} in the setting of double fibration transform introduced in \cite{Helgason}.

The rest of the paper is organized as follows. In Section \ref{sec2}, we establish a decomposition result, which will be used in subsequent sections to prove the main theorems, and we also recall some existing results relevant to our purposes. Section \ref{sec:injectivity with full data} is devoted to uniqueness results and is divided into three subsections. In Subsection \ref{Sub-section3.1}, we prove a global injectivity result (Theorem \ref{thm:whole_inj_MLRT2}) when the direction vector varies over $\Rn \setminus \{0\}$. In Subsection \ref{Sub-section3.2}, we describe the kernel of the MLRT when $\omega$ varies in $\sn$. Moreover, in Subsection \ref{Sub-section3.3}, we provide the kernel description when $\omega$ varies within a neighborhood $B_n(\omega_0, \delta)$ for some fixed $\om_0\in \sn$ and $\delta > 0$. Finally, in Section \ref{sec6}, we present an inversion algorithm to recover vector fields and the trace-free component of higher order symmetric tensor fields using the MLRT restricted to $B_n(\omega_0, \delta)$.

\section{Preliminary results} \label{sec2}

In this section, we collect several preliminary results related to the Light Ray Transform (LRT). Our analysis in this section follows the existing literature, though we highlight and discuss proofs of certain results that involve either mild generalizations or adaptations of previously known arguments.
We begin with the proof of the injectivity of the LRT for functions  with $c>0$. Recall that this result was established in \cite{stefanov1989inverse,ramm1991inverse} for the case $c=1$, utilizing the Fourier Slice theorem.
To this end, we  recall the definition of partial Fourier transform in the hyperplane $(c,\omega)^{\perp}$ as
	\[
	(\F_{(c,\omega)^{\perp}}g)(\zeta)=\int_{(c,\omega)^{\perp}}g(l)e^{-\i l\cdot \zeta }dH(l),
	\]
	where $\i^2=-1$, $dH$ denotes the $n$ dimensional Lebesgue measure on $(c,\omega)^{\perp}$.
	\begin{lemma}[Fourier slice theorem, \cite{ramm1991inverse}] \label{FST}
		For every $f \in L^{1}(\R^{1+n})$ and each fixed $\omega \in \S^{n-1}$, 
		\[
		\hat{f}(\zeta)=\sqrt{2}\F_{(c,\omega)^{\perp}}(\tilde{L}^{0,0}f(-,\omega))(\zeta), \quad\text{whenever}\quad \zeta \perp (c,\omega). 
		\]
	\end{lemma}
	This is a restricted version of the classical Fourier slice theorem. As a corollary one can prove the uniqueness result for functions.  Given that, $\Tilde{L}^{0,0}f((t,x),\omega)=0$ for all $(t,x)\in \mathbb{R}^{1+n}$ and $\omega \in B_{n}(\omega_0, \epsilon)$, using Lemma \ref{FST}, we deduce $\hat{f}(\zeta)=0$ for $\zeta \perp (c,\omega)$ where $\omega \in B_{n}(\omega_0, \epsilon)$. Now varying $\omega$ in $B_n(\omega_0, \epsilon)$ one can show that the  $\{\zeta \in \R^{1+n} : \zeta \perp (c,\om),\, \om \in B_n(\om_0,\epsilon)\}$ has non-empty interior in $\Rno$, $n\geq 2$.  On the other hand, as $f$ is a compactly supported function, the Paley-Wiener theorem implies $\hat{f}$ is real analytic. As a result, $\hat{f}=0$ everywhere. This completes the proof.

	Our next result shows that the highest order MLRT determines all the lower order MLRTs.
	\begin{lemma}\label{lem:lower_mlrt_detremination}
For any $f^{(m)}  \in C_{c}^{\infty}(\mathbb{R}^{1+n};S^m)$, one has 
\begin{align*}
	\langle \Tilde{\omega} \; , \; \nabla_{t,x} \rangle ^{p}\tilde{L}^{m,k}f^{(m)}((t,x),\om) &=(-1)^p\frac{k!}{(k-p)!}\tilde{L}^{m,k-p}f^{(m)}((t,x),\om), \qum{for $0\le p\le k$}.
\end{align*}
	\end{lemma}
			\begin{proof}
			For $p\leq k$ and $\om \in \sn$, the following identity holds
			\begin{equation*}
				\begin{aligned}
					\langle \Tilde{\omega} , \nabla_{t,x} \rangle ^{p}L^{m,k}f^{(m)}((t,x),\om )&=\sum_{i_1,\dots,i_m=0}^{n}\tilde{\omega}_{i_1}\cdots \tilde{\omega}_{i_m}\langle \Tilde{\omega}  , \nabla_{t,x} \rangle ^{p-1}\int_{\mathbb{R}}s^k\langle \tilde{\omega},\nabla_{t,x} \rangle f_{i_1\dots i_m}^{(m)}((t,x)+s\Tilde{\omega})\,ds\\
					&=\sum_{i_1,\dots,i_m=0}^{n}\tilde{\omega}_{i_1}\cdots \tilde{\omega}_{i_m}\langle \Tilde{\omega} \; , \; \nabla_{t,x} \rangle ^{p-1}\int_{\mathbb{R}}s^k \frac{d}{ds}\left[f_{i_1\dots i_m}^{(m)}((t,x)+s\Tilde{\omega})\right]\, ds\\
					&=-\sum_{i_1,\dots,i_m=0}^{n}\tilde{\omega}_{i_1}\cdots \tilde{\omega}_{i_m}\langle \Tilde{\omega} \; , \; \nabla_{t,x} \rangle ^{p-1}k\int_{\mathbb{R}}s^{k-1}f_{i_1\dots i_m}^{(m)}((t,x)+s\Tilde{\omega})\,ds.\\
				\end{aligned}
			\end{equation*}
			We get last equality using integration by parts and the fact that components of $f^{(m)}$ are compactly supported. As $p\leq k$, we can continue this process for $k$-times. After performing this process for $p$ times, we get
			\begin{equation*}
				\begin{aligned}
					\langle \Tilde{\omega} \; , \; \nabla_{t,x} \rangle ^{p}L^{m,k}f^{(m)}((t,x),\om)&=(-1)^p \sum_{i_1,\dots,i_m=0}^{n}\tilde{\omega}_{i_1}\cdots \tilde{\omega}_{i_m}\frac{k!}{(k-p)!}\int_{\mathbb{R}}s^{k-p}f_{i_1\dots i_m}^{(m)}((t,x)+s\Tilde{\omega})ds\\
					&=(-1)^p\frac{k!}{(k-p)!}L^{m,k-p}f^{(m)}((t,x),\om).
				\end{aligned}
			\end{equation*}
		This completes the proof.
		\end{proof}

	Now we discuss an important aspect of symmetric tensors, which resembles the trace-free decomposition. This particular decomposition will be useful in the rest of this article. Before we demonstrate the decomposition, let's define the following operator $J:S^m\to S^{m-2}, $ by 
	\[	[Jf^{(m)}]_{i_1\cdots i_{m-2}}=-c^2f^{(m)}_{i_1\dots i_{m-2}00}+\sum_{p=1}^nf^{(m)}_{i_1\cdots i_{m-2}pp}.
	\]
	We now state a decomposition result using $J$. A somewhat similar decomposition result can be found in \cite{Dairbekov_Sharafutdinov,Paternain_Salo_Uhlmann_2023}.
	\begin{proposition}\label{decompose}
	Assume that, $f^{(m)}\in S^m$ for $m\geq 2$. Then, $f^{(m)}$ can be decomposed in the following way
	\begin{equation}\label{eq_dec}
	f^{(m)}=A^{(m)}+i_{g_{\frac{1}{c}}}f^{(m-2)}, \text{ where } JA^{(m)}=0.
	\end{equation}
	\end{proposition}
	The proof of this decomposition result is based on the following lemma. 
	\begin{lemma}\label{lem_commutator}
Let $n\ge 2$ and   $f^{(m-2)} \in S^{m-2}$ for $m\ge 2$, then the following relations hold.
\begin{itemize}
\item [1.]	\[J\gc f^{(m-2)}= D_{m}\gc Jf^{(m-2)}+ C_mf^{(m-2)},\] 
where $D_m=\binom{m}{2}^{-1}\binom{m-2}{2},$ and $C_m=\binom{m}{2}^{-1}(n+m-1)$.
\item [2.] \[\gc f^{(m-2)}=0 \implies  f^{(m-2)}=0.\]
\end{itemize}
	\end{lemma}
	\begin{proof}
		 For fixed indices $i_1,i_2,\dots,i_{m-2} = 0,\cdots,n$,
		\begin{equation*}
			\begin{aligned}
				\Big[J\gc f^{(m-2)} \Big]_{i_1\dots i_{m-2}}&=-c^2\Big[\gc f^{(m-2)} \Big]_{i_1\dots i_{m-2}00}+\sum_{k=1}^n\Big[\gc f^{(m-2)} \Big]_{i_1\dots i_{m-2}kk}\\
				&=-c^2\binom{m}{2}^{-1}\underbrace{\Big[g_{i_1i_2}f^{(m-2)}_{i_3\dots i_{m-2}00}+\cdots +g_{i_{m-1}i_m}f^{(m-2)}_{i_1\dots i_{m-2}} \Big]}_{i_{m-1}=i_m=0}\\
				&\quad+\sum_{k=1}^{n}\binom{m}{2}^{-1}\underbrace{\Big[g_{i_1i_2}f^{(m-2)}_{i_3\dots i_{m-2}kk}+\cdots +g_{i_{m-1}i_m}f^{(m-2)}_{i_1\dots i_{m-2}} \Big]}_{i_{m-1}=i_m=k}.
			\end{aligned}
		\end{equation*}
		
The last term in the first bracket is equal to 
\[
-c^2 \binom{m}{2}^{-1} g_{00} f^{(m-2)}_{i_1 \dots i_{m-2}}
= \binom{m}{2}^{-1} f^{(m-2)}_{i_1 \dots i_{m-2}}.
\]
A similar calculation shows that the last term in each of the other brackets is also equal to 
		\(\binom{m}{2}^{-1} f^{(m-2)}_{i_1 \dots i_{m-2}}\). 
		Therefore, these terms contribute a total of 
		\((n+1)\binom{m}{2}^{-1} f^{(m-2)}\).
		Next, we carefully handle the remaining terms in each bracket.  
		First, consider those terms where \(i_{m-1} = i_m = 0\), so that they remain with \(f^{(m-2)}\). 
		In the first bracket, their contribution is 
		\[
		-c^2 \binom{m}{2}^{-1} \Big[
		g_{i_1 i_2}(f^{(m-2)}_{00})_{i_3 \dots i_{m-2}} 
		+ \cdots 
		+ g_{i_{m-3} i_{m-2}}(f^{(m-2)}_{00})_{i_1 \dots i_{m-4}}
		\Big].
		\]
		For the other brackets, the same reasoning yields 
		\[
		\sum_{k=1}^{n} \binom{m}{2}^{-1} \Big[
		g_{i_1 i_2}(f^{(m-2)}_{kk})_{i_3 \dots i_{m-2}} 
		+ \cdots 
		+ g_{i_{m-3} i_{m-2}}(f^{(m-2)}_{kk})_{i_1 \dots i_{m-4}}
		\Big].
		\]
		Together, these two expressions contribute 
		\[
		\binom{m-2}{2} \binom{m}{2}^{-1} 
		\big[ \gc J f^{(m-2)} \big]_{i_1 \dots i_{m-2}}.
		\]
		
		Now we are left with the following terms:
		\[
		\begin{aligned}
			&-c^2 \binom{m}{2}^{-1} 
			\Big[ g_{i_1 0} f^{(m-2)}_{i_2 \dots i_{m-2} 0} 
			+ \cdots 
			+ g_{i_{m-2} 0} f^{(m-2)}_{i_1 \dots i_{m-3} 0} \Big] 
			+ \cdots \\ 
			&\quad + \sum_{k=1}^{n} \binom{m}{2}^{-1} 
			\Big[ g_{i_1 k} f^{(m-2)}_{i_2 \dots i_{m-2} k} 
			+ \cdots 
			+ g_{i_{m-2} k} f^{(m-2)}_{i_1 \dots i_{m-3} k} \Big].
		\end{aligned}
		\]
		
		Rewriting, we obtain
		\[
		\begin{aligned}
			&= \binom{m}{2}^{-1} \Big[ -c^2 g_{i_1 0} f^{(m-2)}_{i_2 \dots i_{m-2} 0} 
			+ \sum_{k=1}^n g_{i_1 k} f^{(m-2)}_{i_2 \dots i_{m-2} k} \Big] + \cdots \\
			&\quad + \binom{m}{2}^{-1} \Big[ -c^2 g_{i_{m-2} 0} f^{(m-2)}_{i_1 \dots i_{m-3} 0} 
			+ \sum_{k=1}^n g_{i_{m-2} k} f^{(m-2)}_{i_1 \dots i_{m-3} k} \Big].
		\end{aligned}
		\]
		By the definition of \(g\), in each bracket exactly one term survives, with coefficient \(+1\). Hence, the total contribution is
		\[
		(m-2)\binom{m}{2}^{-1} f^{(m-2)}_{i_1 \dots i_{m-2}}.
		\]
		
		Putting everything together, we obtain
		\[
		J \gc f^{(m-2)} 
		= \binom{m-2}{2}\binom{m}{2}^{-1} \gc J f^{(m-2)} 
		+ \binom{m}{2}^{-1} (m+n-1) \,f^{(m-2)}.
		\]
		This completes the proof of the first part.
		
		For the second part, we prove the claim by considering particular choices of indices.  
		First, let all indices vanish, i.e.\ \(i_1, \dots, i_{m-2} = 0\). Then  
		\[
		\big[ \gc f^{(m-2)} \big]_{0 \cdots 0} = 0 
		\;\;\implies\;\; f^{(m-2)}_{0 \cdots 0} = 0.
		\]
		
		Next, allow exactly one index to be nonzero. Without loss of generality, let \(i_1 \neq 0\) and all others vanish. Then  
		\[
		\big[ \gc f^{(m-2)} \big]_{i_1 0 \cdots 0} = 0 
		\;\;\implies\;\; f^{(m-2)}_{i_1 0 \cdots 0} = 0.
		\]
		By the symmetry of \(f^{(m-2)}\), every component with exactly one nonzero index must vanish. Proceeding inductively, increasing the number of nonzero indices at each step, the same argument shows that all corresponding components vanish. Hence, every component of \(f^{(m-2)}\) is zero, completing the proof.
	\end{proof}
With this lemma at hand, we are now ready to present the proof of Proposition \ref{decompose}.
		\begin{proof} [Proof of Proposition \ref{decompose}]
			What follows is intended not merely as a formal proof of the decomposition, but rather as a constructive procedure that explicitly demonstrates how it can be achieved. With this perspective in mind, let us begin with
			the assumption that the decomposition \eqref{decompose} can be done. So we can rewrite $f^{(m)}$ in the following form
			\[ \label{1st}
			f^{(m)}=A^{(m)}+i_{g_{\frac{1}{c}}}\Big[\underbrace{ A^{(m-2)}+i_{g_{\frac{1}{c}}}A^{(m-4)}+\cdots+i_{g_{\frac{1}{c}}}A^{(m-2k+1)}+i_{g_{\frac{1}{c}}}^{k-1}f^{(m-2k)}}_{B^{(m-2)}}\Big], \text{ where }\, k=[\frac{m}{2}].
			\]
			We now utilize the condition $JA^{(m)}=0$, which implies $Jf^{(m)}=Ji_{g_{\frac{1}{c}}}B^{(m-2)}.$ This along with the commutator relation stated in Lemma \ref{lem_commutator} implies
			\[
			\begin{aligned}
				Jf^{(m)}=&Ji_{g_{\frac{1}{c}}}B^{(m-2)}
				=&D_{m}i_{g_{\frac{1}{c}}}JB^{(m-2)}+C_mB^{(m-2)}.
			\end{aligned}
			\]
	Let us consider $C_mA^{(m-2)}=Jf^{(m)}$. Note that, for symmetric tensor fields of rank 2 and 3, we are done at this step. Now for $m\geq 4$  from above equation we obtain 
			\[
			\begin{aligned}
				0=&D_mi_{g_{\frac{1}{c}}}J\Big[\underbrace{ A^{(m-2)}+i_{g_{\frac{1}{c}}}A^{(m-4)}+\cdots+i_{g_{\frac{1}{c}}}^{k-1}f^{(m-2k)}}_{B^{(m-2)}}\Big]+C_{m}i_{g_{\frac{1}{c}}}\Big[\underbrace{ A^{(m-4)}+\cdots+i_{g_{\frac{1}{c}}}^{k-2}f^{(m-2k)}}_{B^{(m-4)}}\Big].
			\end{aligned}
			\]
			Next using  the second relation of Lemma \ref{lem_commutator} $i.e., \gc f^{(m-2)}=0 \implies  f^{(m-2)}=0$ from above we conclude that 
			\begin{align*}
				0 =& D_m\,JA^{(m-2)}+D_mJi_{g_{\frac{1}{c}}}{B^{(m-4)}}+C_mB^{(m-4)}\\
					=&D_mJA^{(m-2)}+D'_{m}i_{g_{\frac{1}{c}}}JB^{(m-4)}+ C_m' B^{(m-4)}.
			\end{align*} 
			where $D_m'$ and $C_m'$ are some constant depends on $m$ and $n$. Next we choose  $A^{(m-4)}$ in such way that  $D_m JA^{(m-2)}+C_m' A^{(m-4)}=0$. Again notice that, at this point we are done for symmetric tensor fields of rank 4 and 5. We shall proceed for $m\geq6$. This gives
			\[
			\begin{aligned}
				0=&D_m'i_{g_{\frac{1}{c}}}J B^{(m-4)}+C_m' i_{g_{\frac{1}{c}}}\Big[ \underbrace{A^{(m-6)}+\cdots +i_{g_{\frac{1}{c}}}^{k-3}f^{(m-2k)}}_{B^{(m-6)}}\Big].
			\end{aligned}
			\]
			We now iterate this process and apply Lemma \ref{lem_commutator} to obtain   $A^{(m-2p)}$, for $p=3,\cdots,(k-1)$ until we end up with either $f^{(1)}$ or $f^{(0)}$ depending on whether $m$ is odd or even. Substituting these forms of $A^{(m-2p)}$, for $p=1,\cdots,(k-1)$ and $f^{(0)}$ or $f^{(1)}$ in \eqref{1st}, we reach our desired decomposition \eqref{eq_dec}. The proof is complete.
		\end{proof}
	\begin{remark}
One might ask how this decomposition differs from the usual trace-free decomposition of a symmetric tensor field, to be specific, whether \( J \) and \( j_{g_{\frac{1}{c}}} \) are equal. The answer is negative; however, \( J \) and \( j_{g_{\frac{1}{c}}} \) do coincide in the special case \( c = 1 \).
	\end{remark}

	\section{Injectivity of MLRT}\label{sec:injectivity with full data}
	This section is divided into three parts. First part contains injectivity of MLRT when the direction vectors $\om$ are taken from $(\R^{n}\setminus\{0\})$. The second part demonstrates the injectivity of MLRT provided the direction vectors vary in $\sn$. Finally, the third part studies MLRT when the direction vectors are in an open neighborhood of a fixed direction in $\sn$. For the first case, we prove that the highest moment MLRT of any symmetric tensor field is injective when there is no restriction on the direction vectors. In the second part, MLRT is no longer injective for higher order symmetric tensors. We determine a necessary and sufficient condition for MLRT to be injective when the direction vectors are in $\sn$. For the last part, we prove that the kernel for MLRT coincides with the kernel in the second case for dimension $n \geq 3$, though for $n=2$ they are different.
	
	\subsection{Direction vector in $(\R^n\setminus\{0\})$}\label{Sub-section3.1}
	Here, we study injectivity results of MLRT for symmetric $m$ tensor-fields when the direction vector $\omega$ varies in $\mathbb{R}^{n}\setminus\{0\}$. This allows us to differentiate MLRT with respect to $\omega$. Similar ideas for MRT can  be found in \cite{krishnan2018momentum}.
	
	\begin{lemma}\label{gen1}
		Let $m\geq 1$, $n\geq 2$, $f^{(m)} \in C_{c}^{\infty}(\mathbb{R}^{1+n};S^m) $ and fix an index $p \in \{1,\dots,n\}$. Then the following equality holds
		\begin{equation}\label{rankreducer_1}
			mL^{m-1,k}(f^{(m)})_p =\partial_{\omega_p}L^{m,k}f^{(m)}-\partial_{x_p}L^{m,k+1}f^{(m)} \quad \text{for} \quad k=0,1,\dots,m-1,
		\end{equation}
		where $(f^{(m)})_p$ is a symmetric tensor field of rank $(m-1)$ defined as follows
		\[
		((f^{(m)})_p)_{i_1\dots i_{m-1}}=f^{(m)}_{i_1\dots i_{m-1}p}, \quad \mbox{for} \quad i_1,\cdots,i_{m-1} = 0,\cdots,n.
		\]
	\end{lemma}
		\begin{proof}
			For any $(x,\omega)\in \Rn \times (\Rn\setminus \{0\})$, we have that $\PD_{\omega_p} (x+s\omega)= s\, \PD_{x_p} (x+s\omega)$. Using this we obtain 
			\begin{align*}
				\partial_{\omega_p}L^{m,k}f^{(m)} &= \PD_{\omega_p} \left( \sum_{i_{1},\dots,i_{m}=0}^{n}\Tilde{\omega}_{i_{1}}\cdots \Tilde{\omega}_{i_{m}}\int_{\mathbb{R}} s^k\,f_{i_{1}\dots i_{m}}^{(m)}((t,x)+s\Tilde{\omega})\D s\right)\\
				&= \sum_{i_1,\dots,i_m=0}^{n}\sum_{j=1}^m\tilde{\om}_{i_1}\cdots \delta_{i_j}^p\cdots\tilde{\om}_{i_m}\int_{\R}s^k\Big((f^{(m)})_p\Big)_{i_1\dots i_{j-1}i_{j+1}\dots i_m}((t,x)+s\tilde{\om})ds\\&+ \sum_{i_1,\dots,i_m=0}^{n}\tilde{\om}_{i_1}\cdots \tilde{\om}_{i_m}\int_{\R}s^{k+1}\PD_{x_p}f^{(m
					)}_{i_1\dots i_m}((t,x)+s\tilde{\om})ds\\
				&= m L^{m-1,k} (f^{(m)})_p+\PD_{x_p} L^{m,k+1} f^{(m)}.
			\end{align*}
			This completes the proof.
		\end{proof}
	
	\begin{proof}[\textbf{Proof of Theorem \ref{thm:whole_inj_MLRT2}}]
		For $m=0$, the result is classical, see \cite{stefanov1989inverse,ramm1991inverse}. For $m=1$ by assumption, we have that $L^{1,1} f^{(1)}=0$. This along with Lemma \ref{lem:lower_mlrt_detremination} implies that $ L^{1,k} f^{(1)}=0$ for $k=0,1$. Next using Lemma \ref{gen1} we obtain $L^{0,0}(f^{(1)})_p=0$ for $1\le p\le n$. This along with injectivity of the LRT over scalar functions implies $(f^{(1)})_p=0$ for $1\le p\le n$. Inserting this into $ L^{1,0} f^{(1)}=0$ entails $L^{0,0} (f^{(1)})_0=0$, which further gives $ (f^{(1)})_0=0$. As a result, we have $f^{(1)}=0$.
		
	We next argue by induction and assume that the result is true for $m-1\ge 0$.  Then for $m$, by \eqref{rankreducer_1}, we obtain \( L^{m-1,k}(f^{(m)})_p = 0 \) for \( k = 0, 1, \ldots, m-1 \), and $1\le p\le m$. By induction we have that $(f^{(m)})_p =0$ for $1\le p\le n$. Using the symmetry of $f^{(m)}$, we have that each element of  $f^{(m)}$ is zero except $ f^{(m)}_{00 \cdots 0}$. This can be achieved using the fact that $L^{m,0}f^{(m)}=0$ and  $(f^{(m)})_p =0$ for $1\le p\le n$. This completes the proof. 
	\end{proof}
	\subsection{Direction vector restricted on the sphere:}\label{Sub-section3.2}
	In this subsection, we study the injectivity of MLRT when the direction vector $\om$ varies on the sphere. We see that, the kernel description of MLRT of symmetric tensor fields changes significantly from the case where $\om \in \R^{n}\backslash \{0 \}$. We start with the observation that, symmetric tensor products of a symmetric tensor fields with the Minkowski metric tensor ($g_\frac{1}{c}$) always lie in the kernel. The following lemma describes it properly.
	\begin{lemma}\label{backward}
	Let $m\ge 2$ and $n\ge 2$. 	 Then the following equality holds
		\[  
		\tilde{L}^{m,k}\Big(i_{g_{\frac{1}{c}}}f^{(m-2)}\Big)=0 \qum{for any} \quad f^{(m-2)}\in C_c^{\infty}(\R^{1
			+n};S^{m-2}).
		\]
		\begin{proof}
			Consider $(t,x) \in \R^{1+n}$ and $\om \in \S^{n-1}$
			\begin{equation}
				\begin{split}
					\tilde{L}^{m,k}(i_{g_\frac{1}{c}}f^{(m-2)})((t,x),\om) &=\sum_{i_1,\dots , i_m=0}^{n}\tilde{\om}_{i_1}\cdots \tilde{\om}_{i_m}\int_{\R}s^k (i_{g_{\frac{1}{c}}}
					f^{(m-2)})_{i_1\cdots i_m}((t,x)+s\tilde{\om})ds\\
					& =\sum_{i_1,\dots , i_m=0}^{n}\tilde{\om}_{i_1}\cdots \tilde{\om}_{i_m}\int_{\R}s^k \sum_{\pi \in \Pi_{m}} f^{(m-2)}_{i_{\pi(1)}\cdots i_{\pi(m-2)}}(g_{\frac{1}{c}})_{i_{\pi(m-1)}i_{\pi(m)}}ds\\
					&=\sum_{i_1,\dots , i_m=0}^{n}\tilde{\om}_{i_1}\cdots \tilde{\om}_{i_m} \sum_{\pi \in \Pi_{m}}(g_{\frac{1}{c}})_{i_{\pi(m-1)}i_{\pi(m)}}\int_{\R}s^kf^{(m-2)}_{i_{\pi(1)}\cdots i_{\pi(m-2)}}ds.
				\end{split}
			\end{equation}
			For any fixed $\pi \in \Pi_m$, $\pi(m-1), \pi(m)$ will be in $\{1,2,\dots,m\}$. Without loss of generality assume that, $	i_{\pi(m-1)}=i_1;\quad i_{\pi(m)}=i_2.$ This gives
			\[
			\tilde{L}^{m,k}(i_{g_\frac{1}{c}}f^{(m-2)})((t,x),\om) =\sum_{\pi \in \Pi_m}\sum_{i_3,\dots , i_m=0}^{n}\tilde{\om}_{i_3}\cdots \tilde{\om}_{i_m}[\sum_{i_1,i_2=0}^{n}\tilde{\om}_{i_1}\tilde{\om}_{i_2}(g_{\frac{1}{c}})_{i_1i_2}]\int_{\R}s^kf^{(m-2)}_{i_{\pi(1)}\cdots i_{\pi(m-2)}}ds.
			\]
			Recall that $\tilde{\omega}=(c, \omega)$, this implies  $\sum_{i_1,i_2=0}^{n}\tilde{\om}_{i_1}\tilde{\om}_{i_2}(g_{\frac{1}{c}})_{i_1i_2}=0.$ Therefore, 
            \[\tilde{L}^{m,k}(i_{g_\frac{1}{c}}f^{(m-2)})((t,x),\om) =0.\] 
		\end{proof}
	\end{lemma}
	The above lemma provides proof of the converse parts of both Theorem \ref{thm_fullsph} and Theorem \ref{thm_partsph}.
	\begin{proof}[\textbf{Proof of Theorem \ref{thm_fullsph}}]
		\,Let $\widehat{f}$ denotes the Fourier transform of $f$ component-wise in  the $t$ variable.
		By definition of the Fourier transform,
		\begin{align*}
			\widehat{\tilde{L}^{m,k}f^{(m)}}(\tau,x,\omega)
			&= \sum_{i_1,\dots ,i_m=0}^{n}\int\limits_{\mathbb{R}} \int\limits_{\mathbb{R}} s^k\, e^{-\i t \tau} f^{(m)}_{i_1\cdots i_m}(t+sc,x+s\omega) \,\Tilde{\omega}_{i_1}\cdots \Tilde{\omega}_{i_m} ds\,  dt\\
			&=\sum_{i_1,\dots ,i_m=0}^{n} \int\limits_{\mathbb{R}} s^k\,\int\limits_{\mathbb{R}} e^{-\i (r-sc) \tau} f^{(m)}_{i_1\cdots i_m}(r,x+s\omega) \,\Tilde{\omega}_{i_1}\cdots \Tilde{\omega}_{i_m} dr\,  ds\\
			&=\sum_{i_1,\dots ,i_m=0}^{n} \int\limits_{\mathbb{R}} s^k e^{\i sc \tau} \widehat{f^{(m)}_{i_1\cdots i_m}}(\tau,x+s\omega) \,\Tilde{\omega}_{i_1}\cdots \Tilde{\omega}_{i_m}  ds.
		\end{align*}
	Thus, we arrive at
		\begin{equation}\label{main_relation}
			\widehat{\tilde{L}^{m,k}f^{(m)}}(\tau,x,\omega)=\sum_{i_1,\dots ,i_m=0}^{n}\int\limits_{\mathbb{R}} s^k\, e^{\i sc \tau} \widehat{f^{(m)}_{i_1\cdots i_m}}(\tau,x+s\omega)\, \Tilde{\omega}_{i_1}\cdots \Tilde{\omega}_{i_m}  ds.
		\end{equation}
		We use the decomposition result from  Lemma \ref{decompose} to write $f^{(m)} = A+i_{g_{\frac{1}{c}}}f^{(m-2)}$ where $JA=0$, and deduce
		\begin{align}\label{Fourier_tra_l_g}
			\widehat{\tilde{L}^{m,k}f^{(m)}}(\tau,x,\omega)=\widehat{\tilde{L}^{m,k}A}(\tau,x,\omega)=\sum_{i_1,\dots ,i_m=0}^{n}\int\limits_{\mathbb{R}} s^k\, e^{\i sc \tau} \hat{A}_{i_1\cdots i_m}(\tau,x+s\omega) \,\Tilde{\omega}_{i_1}\cdots \Tilde{\omega}_{i_m}  ds .
		\end{align}
		Since $\tilde{L}^{m,k}f^{(m)} =0$ implies $ \tilde{L}^{m,k}A=0$, therefore, $\widehat{\tilde{L}^{m,k}A}(0,x,\omega)=0$ for $x\in \R^{n}$, $\omega \in \S^{n-1}$.
		Now, using symmetry of $A$, we rewrite the identity $\tilde{L}^{m,k}A=0$ as
		\begin{equation}\label{step_1}
			\begin{aligned}
				0= \widehat{\tilde{L}^{m,k}A}(0,x,\omega)
				&=\sum_{i_1,\dots ,i_m=0}^{n}\int\limits_{\mathbb{R}} s^k\, \hat{A}_{i_1\cdots i_m}(0,x+s\omega)\, \Tilde{\omega}_{i_1}\cdots \Tilde{\omega}_{i_m}  ds\\
				&=\sum\limits_{l=0}^m\sum_{i_1,\dots ,i_l=1}^{n} \int\limits_{\mathbb{R}}c^{m-l} s^k\,  \binom{m}{l}\,\hat{A}^{(l)}_{i_1 \cdots i_l}(0,x+s\omega)\,\omega_{i_1}\cdots \omega_{i_l} \, ds.
		\end{aligned}\end{equation}
		Here $\hat{A}^{(l)}_{i_1 \cdots i_l}:= \hat{A}_{i_1\cdots i_l\underbrace{0\dots 0}_{m-l} }$ is a symmetric $l$ tensor in $\mathbb{R}^n$, and for $l=0$, $\hat{A}^{(0)} = \hat{A}_{0\cdots 0}$ is a function. To this end, we define two operators $i_{\delta}:S^m\rightarrow S^{m+2} $ and $j_{\delta}:S^{m}\rightarrow S^{m-2}$ as follows:
\begin{align}
(i_{\delta}f )_{i_1 \cdots i_{m+2}}&\coloneqq \sigma(f_{i_1\cdots i_m} \otimes \delta_{i_{m+1}i_{m+2}})\label{def_of_i}\\
(j_{\delta} f)_{i_1\cdots i_{m-2}}&\coloneqq \sum\limits_{k=1}^n f_{i_1\cdots i_{m-2}kk},\label{def_of_j}
\end{align}
where $\sigma$ denotes the symmetrization of a tensor field. Here $\delta_{ij}$ is the Kronecker delta tensor which is equal to $1$ for $i=j$ and $0$ otherwise.
Replacing $\omega$ by $-\omega$ in the above equation and adding/subtracting it with the original one, we obtain
\begin{align}\label{eq_1.5}
0=&  \widehat{\tilde{L}^{m,k}A}(0,x,\omega)\pm \widehat{\tilde{L}^{m,k}A}(0,x,-\omega)\nonumber \\
&=\sum\limits_{l=0}^m \binom{m}{l}c^{m-l}\sum_{i_1,\dots ,i_l=1}^{n} \int\limits_{\mathbb{R}} s^k\, \,\left( \hat{A}_{i_1\cdots i_l}^{(l)} \pm (-1)^{l-k}\hat{A}_{i_1\cdots i_l}^{(l)}\right) (0,x+s\omega)\,\omega_{i_1}\cdots \omega_{i_l} \, ds.
\end{align}
		
In the rest of the proof we consider $m$ is even, say $2p$. The proof for $m$ odd will follow similarly. To proceed further, let us consider the symmetric tensor fields:
		\begin{align*}
			G_{i_1\cdots i_{2p}}^+(0,x)&= \sum\limits_{l=0}^{p} \binom{2p}{2l} c^{2p-2l} i^{p-l}_{\delta}\hat{A}_{i_1\cdots i_{2l}}^{(2l)}(0,x),\\
			G_{i_1\cdots i_{2p-1}}^-(0,x)&=  \sum\limits_{l=1}^{p}  \binom{2p}{2l-1} c^{2p-2l+1}i^{p-l}_{\delta}\hat{A}_{i_1\cdots i_{2l-1}}^{(2l-1)}(0,x).
		\end{align*}
		Now, for each $k=0,1,\cdots,m$, the identity \eqref{eq_1.5} reads
		\begin{align*}
			\int\limits_{\mathbb{R}} s^k\,  G_{i_1\cdots i_{2p}}^+(0,x+s\omega) \,\omega_{i_1}\cdots \omega_{i_{2p}} \, ds&=0, \qquad \text{ [for $(+)$ve sign] },\\
			\int\limits_{\mathbb{R}} s^k\,  G_{i_1\cdots i_{2p-1}}^-(0,x+s\omega) \,\omega_{i_1}\cdots \omega_{i_{2p-1}} \, ds&=0, \qquad \text{ [for $(-)$ve sign] }.
		\end{align*}
		Now injectivity of MRT (see \cite[Theorem 2.17.2]{Sharafutdinov_book} and \cite{krishnan2018momentum,BKS_MRT_polyharmonic}) gives $ G^{\pm}=0$ in $ \mathbb{R}^n$, which implies 
		\begin{align*}
			\sum\limits_{l=0}^{p} \binom{2p}{2l} c^{2p-2l}  i^{p-l}_{\delta}\hat{A}_{i_1\cdots i_{2l}}^{(2l)}=0 \quad \implies \,
			\hat{A}_{i_1\cdots i_{2p}}^{(2p)}=- \sum\limits_{l=0}^{p-1}  \binom{2p}{2l}  c^{2p-2l} i^{p-l}_{\delta}\hat{A}_{i_1\cdots i_{2l}}^{(2l)}.
		\end{align*}
		Now we consider the following inner product 
		\begin{equation}\label{eq_inner_prod}
			\begin{aligned}
				-\left\langle \hat{A}_{i_1\cdots i_{2p}}^{(2p)} , \hat{A}_{i_1\cdots i_{2p}}^{(2p)} \right\rangle
				&=\left\langle \hat{A}_{i_1\cdots i_{2p}}^{(2p)} ,  \sum\limits_{l=0}^{p-1} \binom{2p}{2l} c^{2(p-l)} i^{p-l}_{\delta}\hat{A}_{i_1\cdots i_{2l}}^{(2l)} \right\rangle
				\\&= \sum\limits_{l=0}^{p-1}\left\langle j^{p-l}_{\delta} \binom{2p}{2l} c^{2(p-l)} \hat{A}_{i_1\cdots i_{2p}}^{(2p)} ,   \hat{A}_{i_1\cdots i_{2l}}^{(2l)} \right\rangle.
			\end{aligned}
		\end{equation}
		The condition $JA^{(m)}=0$, entails 
		\begin{align*}
			j_{\delta} \hat{A}^{(2p)}_{i_1\cdots i_{2p}}(0,x)=  {c^2} \hat{A}^{(2p-2)}_{i_1\cdots i_{2p-2}}(0,x).
		\end{align*}
		Using the above equality, we finally obtain
		\begin{align}\label{eq_1.9}
			j^{l}_{\delta} \hat{A}^{(2p)}_{i_1\cdots i_{2p}}(0,x) = {c^{2l}} \hat{A}^{(2p-2l)}_{i_1\cdots i_{2p-2l}}(0,x)={c^{2l}}  \hat{A}^{(m-2l)}_{i_1\cdots i_{m-2l}}(0,x).
		\end{align}
		Combining \eqref{eq_inner_prod}, and \eqref{eq_1.9} we obtain
		\begin{equation}\label{step_2}
			\begin{aligned}
				\langle \hat{A}_{i_1\cdots i_{2p}}^{(2p)} , \hat{A}_{i_1\cdots i_{2p}}^{(2p)} \rangle +    \sum\limits_{l=0}^{p-1}
				\binom{2p}{2l}c^{4(p-l)}\langle \hat{A}_{i_1\cdots i_{2l}}^{(2l)} ,\hat{A}_{i_1\cdots i_{2l}}^{(2l)} \rangle&=0.
			\end{aligned}
		\end{equation}
		Since $\binom{2p}{2l}c^{4(p-l)} >0$, the above equation implies \[ \hat{A}^{(2p-2l)}_{i_1\cdots i_{2p-2l}}(0,x)=0,\quad \text{ for } 0\le l \le p.\] Now we repeat the same analysis for $G^-$ and obtain \[ \hat{A}^{(2p-2l+1)}_{i_1\cdots i_{2p-2l+1}}(0,x)=0 \quad \text{ for } 1\le l \le p.\] This implies $ \hat{A}(0,x)=0$. 
		
		We next argue by induction and show that  $ \PD^{n_1}_{\tau} \hat{A}_{i_1\cdots i_m}(0,x)$ $=0 $ for all $n_1\in \mathbb{N}$.
		Differentiating \eqref{Fourier_tra_l_g} with respect to $\tau$ for $n_1$ times, we obtain
		\begin{align}\label{p_derivative_on_lg_hat}
			\PD^{n_1}_{\tau}  \widehat{\tilde{L}^{m,k}A}(\tau,x,\omega)= \sum\limits_{l=0}^{n_1} \int\limits_{\mathbb{R}} s^k\,  (\i sc )^{n_1-l} \binom{n_1}{l}\, e^{\i sc \tau}\, \PD^l_{\tau} \hat{A}_{i_1\cdots i_m}(\tau,x+s\omega)\, \Tilde{\omega}_{i_1}\cdots \Tilde{\omega}_{i_m}  ds. 
		\end{align}
		By assumption of Theorem \ref{thm_fullsph} we have $ \PD^{n_1}_{\tau}  \widehat{\tilde{L}^{m,k}A}(\tau,x,\omega)|_{\tau=0}=0$ for all $n_1\in \Nt$. 
		For $n_1=1$ \eqref{p_derivative_on_lg_hat} implies 
		\begin{align*}
			0= \PD^{1}_{\tau}  \widehat{\tilde{L}^{m,k}A}(0,x,\omega)
			&= \int\limits_{\mathbb{R}} s^k\,    e^{\i sc \tau}\, \PD^1_{\tau} \hat{A}_{i_1\cdots i_m}(0,x+s\omega) \,\Tilde{\omega}_{i_1}\cdots \Tilde{\omega}_{i_m} ds.
		\end{align*}
		Repeating the analysis in \eqref{step_1}-- \eqref{step_2} for $\hat{A}_{i_1\cdots i_m}(0,\cdot)$ replaced by $\PD^1_{\tau}\hat{A}_{i_1\cdots i_m}(0,\cdot)$, we obtain $\PD_{\tau} \hat{A}(0,x)=0$, for all $x\in \R^n$. So our statement is true for $n_1=1$.
		Now assume that, $\PD^{l}_{\tau}\hat{A}(0,x)$ $=0$ for $l\leq (n_1-1)$ and for all $x\in \R^n$, then from \eqref{p_derivative_on_lg_hat} we obtain
		\begin{align*}
			0= \PD^{n_1}_{\tau}  \widehat{\tilde{L}^{m,k}A}(\tau,x,\omega)
			=\int\limits_{\mathbb{R}} s^k\,    e^{\i sc \tau}\, \PD^{n_1}_{\tau} \hat{A}_{i_1\cdots i_m}(0,x+s\omega)\, \Tilde{\omega}_{i_1}\cdots \Tilde{\omega}_{i_m} ds.    
		\end{align*}
		Using \eqref{step_1}-- \eqref{step_2} for $\PD^{n_1}_{\tau}\hat{A}_{i_1\cdots i_m}(0,\cdot)$, the above integral identity entails $\PD^{n_1}_{\tau} \hat{A}_{i_1\cdots i_m}(0,x)=0 $ for all $x\in \R^n$. Thus by induction we have $ \PD^{n_1}_{\tau} \hat{A}_{i_1\cdots i_m}(0,x)$ $=0 $ for all $n_1 \in \Nt$ for all $x\in \R^n$. Since $A$ is compactly supported in $t$ variable, by Paley-Wiener theorem $ \hat{A}(\tau,x)$ is analytic in $\tau$ and $ \PD^{n_1}_{\tau}\hat{A}(0,x)=0$ implies $ \hat{A}(\tau,x)=0$ for every $\tau \in \R$. Thus $A(t,x)=0$ for all $(t,x)\in \R^{1+n}$. Hence $f^{(m)} =i_{g_{\frac{1}{c}}}f^{(m-2)}$. This completes the proof.
	\end{proof}
	\subsection{Direction vector near a fixed  vector on $\sn$:}\label{Sub-section3.3}
	So far, we have proved injectivity results for MLRT of symmetric $m$-tensor fields where the direction vectors lie in \(\mathbb{R}^{n} \setminus \{0\}\) or on $\S^{n-1}$. Now, we restrict our focus to the case when the direction vectors come from an open neighborhood in $\sn$, $n \geq 3$ of a fixed unit vector $\om_0 \in \sn$. Given $\om_0 \in\sn$ and $\delta>0$, recall the set 
	\[
	B_n(\om_0,\delta) = \{ \om \in \sn : |\om-\om_0|<\delta\}.
	\]
	
	\begin{definition}\label{tanggrad}\cite{pressley2010elementary}
		Let \( S \subset \mathbb{R}^{n} \) be an open and bounded set with a smooth boundary denoted by \( \Gamma \), and let \( f : \Gamma \to \mathbb{R} \) be a smooth function on \( \Gamma \). The tangential gradient of \( f \) on the surface \( \Gamma \) is defined by
		\begin{equation}
			\nabla_{\Gamma} f \coloneqq \left[\nabla F - \left( \frac{\partial F}{\partial \nu} \right) \nu\right] \Big|_{\Gamma},
		\end{equation}
		where \( F \) is a smooth extension of \( f \) onto \( \mathbb{R}^{n} \), and \( \nu \) is the normal vector to \( \Gamma \). This definition is independent of the choice of the extension \( F \). 
	\end{definition}
	\begin{remark}
			Since in our situation the embedded sub-manifold is $\sn\subset \Rn$, the tangential gradient is given by 
		\begin{equation*}
		\nabla_S f = \nabla_{\omega} F - (\omega\cdot \nabla_{\omega} F)\om, \quad \mbox{where} \quad \omega \in \sn.
		\end{equation*}
		Note that these tangent vectors are not linearly independent as they satisfy the following relation $ \omega\cdot \nabla_S =0$. 
	\end{remark}
	Let us start the discussion for symmetric 2-tensor fields.
	\begin{lemma}\label{lemma2.7}
		Assume that $f^{(2)}\in C^{\infty}_{c}(\mathbb{R}^{1+n};S^{2})$ and $n\geq 2$. If 
		$
		\Tilde{L}^{2,k}f^{(2)}((t,x), \omega) =0 $, for all $ \omega \in \Bdl,\, (t,x) \in \mathbb{R}^{1+n}\,\text{and }\, k = 0,1,2,
		$
		then the following identity holds
		\begin{equation}\label{relation1}
			\Tilde{L}^{1,k}(f^{(2)})_i=-c\omega_{i}\Tilde{L}^{1,k}(f^{(2)})_{0}, \quad\text{for } i=1,2,\dots,n,
		\end{equation} 
		where $(f^{(2)})_{i}=(f^{(2)}_{0i},\dots,f^{(2)}_{ni})$, the i-th column of the associated matrix.
		\begin{proof}
			We start with taking tangential gradient (in the $\om$ variable) of $\Tilde{L}^{2,k}f^{(2)}((t,x),-)$ over $ B_{n}(\omega_0,\delta)$. So the definition of tangential gradient entails
			\begin{equation}
				\nabla_{S}\Tilde{L}^{2,k}f^{(2)}((t,x),\om)=\Big[\nabla_{\omega}L^{2,k}f^{(2)}((t,x),\om)-(\omega \cdot \nabla_{\omega}L^{2,k}f^{(2)}((t,x),\om))\omega\Big ],
			\end{equation}
			where $L^{m,k}f^{(2)}((t,x),-)$ plays the role of extension of $\tilde{L}^{(m,k)}f^{(2)}((t,x),-)$ onto $\R^n\setminus\{0\}$ and  $\nabla_{\omega}$ denotes gradient in $\omega $ variables. Thus the $i$-th component of the tangential gradient is 
			\begin{equation*}
				\begin{aligned}
					(\nabla_{S}\Tilde{L}^{2,k}f^{(2)}((t,x),\om))_i &= \partial_{\omega_i}L^{2,k}f^{(2)}((t,x),\om)-\omega_{i}\Big[\sum_{j=1}^{n}\om_j\partial_{\omega_j}L^{2,k}f^{(2)}((t,x),\om)\Big]\\
					& =2L^{1,k}(f^{(2)})_i((t,x),\om) +\partial_{x_i}L^{2,k+1}f^{(2)}((t,x),\om)\\&\qquad-\omega_{i}\Big [\sum_{j=1}^{n}\omega_j [2L^{1,k}(f^{(2)})_j ((t,x),\om)+\partial_{x_j}L^{2,k+1}f^{(2)}((t,x),\om)]\Big].
				\end{aligned}
			\end{equation*}
We next observe that	
\[  (\nabla_{S}\Tilde{L}^{2,k}f^{(2)}((t,x),\om)) =0,
\quad \mbox{and} \quad 
\PD_t \Tilde{L}^{2,k}f^{(2)}((t,x),\om)=0, 
\quad \mbox{and} \quad
\nabla_x\Tilde{L}^{2,k}f^{(2)}((t,x),\om)=0,\]  since $	\Tilde{L}^{2,k}f^{(2)}((t,x), \omega) =0 $, for all $ \omega \in \Bdl$, and $(t,x) \in \mathbb{R}^{1+n}\,\text{and }\, k = 0,1,2$. This along with preceding equation entails
			\begin{equation*}
				\begin{aligned}
				&	\Tilde{L}^{1,k}(f^{(2)})_i((t,x),\om)=\omega_i \sum_{j=1}^{n}\omega_{j}\Tilde{L}^{1,k}(f^{(2)})_{j}((t,x),\om)\\
					&\quad=\om_i\sum_{i_1,j=0}^n\tilde{\om}_{i_1}\tilde{\om}_j\int_{\R}s^k f^{(2)}_{i_1j}((t,x)+s\tilde{\om})ds - c\, \om_{i}\sum_{i_1=0}^n\tilde{\om}_{i_1}\int_{\R}s^kf^{(2)}_{i_10}((t,x)+s\tilde{\om})ds\\
					&\quad=\om_i \tilde{L}^{2,k}f^{(2)}((t,x),\om)-c\om_i \tilde{L}^{1,k}(f^{(2)})_0((t,x),\om)= -c\om_i \tilde{L}^{1,k}(f^{(2)})_0((t,x),\om).
				\end{aligned}
			\end{equation*} 
			To obtain the last relation we have again used the fact that $ \Tilde{L}^{2,k}f^{(2)}((t,x),\om)=0$ for $k=0,1$. This completes the proof.
		\end{proof}
	\end{lemma}
	\begin{remark}
		The above lemma provides a relation between MLRT of any column and MLRT of the $0$-th column of the associated matrix. We exploit this relation in the upcoming results. Moreover, we see that a similar relation holds for symmetric $m$-tensors as well.
	\end{remark}
	\begin{lemma}\label{preresult_1}
		Let  $f^{(2)}\in C_{c}^{\infty}(\mathbb{R}^{1+n};S^2)$ and assume that 	
        \[
		\Tilde{L}^{2,k}f^{(2)}((t,x),\omega)=0, \quad\text{for all } \omega \in \Bdl,\, (t,x) \in \mathbb{R}^{1+n} \, \text{and } k=0,1,2,
		\]
		then the following identity holds
		\begin{equation}
			\sum_{j=1}^{n}\Tilde{L}^{0,k}(f^{(2)})_{ii}-\sum_{i,j=1}^{n}\omega_{i} \omega_{j}\Tilde{L}^{0,k}(f^{(2)})_{ij}=c(1-n)\Tilde{L}^{1,k}(f^{(2)})_0, \quad\text{for } k=0,1.
		\end{equation}
	\end{lemma}
	\begin{proof}
		From the given hypothesis and Lemma~\ref{lemma2.7}, we have
		\begin{equation}\label{relation1_1}
			\Tilde{L}^{1,k}(f^{(2)})i = -c\,\omega_{i}\Tilde{L}^{1,k}(f^{(2)})_{0},
			\quad \text{for $k=0,1$} \qum{and $1\le i\le n$}.
		\end{equation}
		We now apply the tangential gradient over $B_{n}(\omega_0,\delta)$ to both sides of this relation. To do so, we compute the tangential gradient of the left-hand side and the right-hand side separately, and then compare the results.
We now apply the tangential gradient on the L.H.S of \eqref{relation1_1} and deduce that

\begin{equation*}
\begin{aligned}
&	(	\nabla_{S}(\Tilde{L}^{1,k}(f^{(2)})_{i})_p\\&=\Big[\partial_{\omega_p}(L^{1,k}(f^{(2)})_{i})-\omega_{p}\sum_{j=1}^{n}\omega_{j}\partial_{\omega_j}(L^{1,k}(f^{(2)})_{i})\Big]\\
&= \Tilde{L}^{0,k}(f^{(2)})_{ip}+\partial_{x_p}(\Tilde{L}^{1,k+1}(f^{(2)})_{i})-\omega_{p}\sum_{j=1}^{n}\omega_{j}[\Tilde{L}^{0,k}(f^{(2)})_{ij}+\partial_{x_j}(\Tilde{L}^{1,k+1}(f^{(2)})_{i})]\\
&=\Tilde{L}^{0,k}(f^{(2)})_{ip}-\omega_{i}\partial_{x_p}(\Tilde{L}^{1,k+1}(f^{(2)})_0)-
\omega_p \sum_{j=1}^{n}\omega_{j}[\Tilde{L}^{0,k}(f^{(2)})_{ij}-  \omega_{i}\partial_{x_j}(\Tilde{L}^{1,k+1}(f^{(2)})_0)].
\end{aligned}
\end{equation*}
Furthermore, choosing  $p=i$ in above and taking  sum over $i=1,\cdots,n$ we obtain 
\begin{align}\label{eq_3.1}
			\sum_{i=1}^n 	(	\nabla_{S}(\Tilde{L}^{1,k}(f^{(2)})_{i})_i= \sum_{i=1}^{n}\Tilde{L}^{0,k}(f^{(2)})_{ii}-\sum_{i,j=1}^{n}\omega_{i} \omega_{j}\Tilde{L}^{0,k}(f^{(2)})_{ij}.
		\end{align}
		Similarly applying $\nabla_S$ on the R.H.S of \eqref{relation1_1} we obtain 
		\begin{equation*}
			\begin{aligned}
			-& (\nabla_{S}\,(c\,\omega_{i}\Tilde{L}^{1,k}(f^{(2)})_{0}))_p\\	=&c\Big[\omega_p \sum_{j=1}^{n}\omega_{j}\partial_{\omega_j}(\omega_{i}L^{1,k}(f^{(2)})_0)-\partial_{\omega_p}(\omega_i L^{1,k}(f^{(2)})_0)\Big]\\
				=&c\Big[\omega_p \sum_{j=1}^{n}\omega_j [\delta_i^jL^{1,k}(f^{(2)})_0+\omega_i \partial_{\omega_j}(L^{1,k}(f^{(2)})_0)\Big] \\
				&-\Big[\delta_i^p \Tilde{L}^{1,k}(f^{(2)})_0+\omega_i \partial_{\omega_p}(L^{1,k}(f^{(2)})_0)\Big]\\
		=&c\Big[\omega_p\sum_{j=1}^{n}\omega_j \delta_j^i \Tilde{L}^{1,k}(f^{(2)})_0 -\delta_i^p \Tilde{L}^{1,k}(f^{(2)})_0+\omega_p\sum_{j=1}^{n}\omega_i \omega_j\Tilde{L}^{0,k}(f^{(2)})_{0j}-\omega_i \Tilde{L}^{0,k}(f^{(2)})_{0p}\\
				&+\omega_p \sum_{j=1}^{n}\omega_i \omega_j \partial_{x_j}(\Tilde{L}^{1,k+1}(f^{(2)})_0 )-\omega_i \partial_{x_p}(\Tilde{L}^{1,k+1}(f^{(2)})_0)\Big].
			\end{aligned}
		\end{equation*}
		Next choosing $p=i$ and summing over $i=1,\cdots,n$ we obtain
		\begin{align}\label{eq_3.2}
-\sum_{i=1}^n(\nabla_{S}\, (c\,\omega_{i}\Tilde{L}^{1,k}(f^{(2)})_{0}))_i= c(1-n)\, \Tilde{L}^{1,k}(f^{(2)})_0.
		\end{align}
	We next equate \eqref{eq_3.1} and \eqref{eq_3.2} and conclude  that
		\begin{equation}\label{help1}
			\sum_{i=1}^{n}\Tilde{L}^{0,k}(f^{(2)})_{ii}-\sum_{i,j=1}^{n}\omega_{i} \omega_{j}\Tilde{L}^{0,k}(f^{(2)})_{ij}=c(1-n)\Tilde{L}^{1,k}(f^{(2)})_0, \quad \text{ for } k=0,1.
		\end{equation}
		This completes the proof.
	\end{proof}

    The relation \eqref{help1} provides an essential tool for analyzing the injectivity of the MLRT in this setup and is also useful for the reconstruction algorithm. We are now ready to present the proof of Theorem \ref{thm_partsph}. 
	\begin{proof}[\textbf{Proof of Theorem \ref{thm_partsph}}]
		The proof of the sufficient part is guaranteed by Lemma \ref{backward}. So we proceed with the proof of the forward part. We divide the proof into several steps.\smallskip
		
		\textbf{Step 1.} The case for $m=0$ and $m=1$.
	
		Given that, $\Tilde{L}^{0,0}f((t,x),\omega)=0$ for all $(t,x)\in \mathbb{R}^{1+n}$ and $\omega \in B_{n}(\omega_0, \delta)$, using the result \ref{FST}, we get $\hat{f}(\zeta)=0$ for $\zeta \perp (c,\omega)$ where $\omega \in B_{n}(\omega_0, \delta)$. It is easy to show that the set $\{\zeta \in \R^{1+n} : \zeta \perp (c,\om),\, \om \in B_n(\om_0,\delta)\}$ is an open set. On the other hand, as $f$ is a compactly supported function, the Paley-Wiener theorem implies $\hat{f}$ is real analytic. Therefore, $\hat{f}=0$, which completes the proof for $m=0$. For $m=1$ we have the $f^{(1)}= \nabla \phi$, since $ \Tilde{L}^{1,0}f^{(1)}=0$, see \cite{krishnan2020inverse}. Next we substitute this form of $f^{(1)}$ in $\Tilde{L}^{1,1}f^{(1)}=0$, we shall get $\Tilde{L}^{0,0}\phi=0$. This together with \cite[Proposition 2.2]{ramm1991inverse} implies $\phi=0$. Consequently, $f^{(1)}=0$. \smallskip
		
		\textbf{Step 2.} The case of $m=2$.
		
		In this case we first assume that $Jf^{(2)}=0$, equivalently $c^2f^{(2)}_{00}=\sum_{i=1}^{n}f^{(2)}_{ii}.$ This along with Lemma \ref{preresult_1} entails
		\begin{align}\label{eq_3.4}
			c^2\Tilde{L}^{0,k}(f^{(2)})_{00}-\sum_{i,j=1}^{n}\omega_i \omega_j\Tilde{L}^{0,k}(f^{(2)})_{ij}= c(1-n)\Tilde{L}^{1,k}(f^{(2)})_0,\quad \text{for }k=0,1.
		\end{align}
		Next using $	0=\tilde{L}^{2,k} f^{(2)}$  we observe that	\begin{align*}
			0=	\tilde{L}^{2,k} f^{(2)}&= \int_\R \sum_{i,j=0}^{n} s^k \tilde{\omega}_i\tilde{\omega}_j f^{(2)}_{ij} ((t,x) +s\tilde{\omega})\, ds, \qum{where $\tilde{\omega}=(c,\omega)$}\\
			&= -\int_\R s^k c^2 f^{(2)}_{00} ((t,x) +s\tilde{\omega})\, ds + 2 c \int_\R\sum_{j=0}^n s^k \tilde{\omega}_j f^{(2)}_{ij} ((t,x) +s\tilde{\omega})\, ds \\&\quad
			+  \int_\R \sum_{i,j=1}^{n} s^k  \omega_i\omega_j \, f^{(2)}_{ij} ((t,x) +s\tilde{\omega})\, ds\\
			&= 	-c^2\Tilde{L}^{0,k}(f^{(2)})_{00} +2c 	\Tilde{L}^{1,k}(f^{(2)})_{0} + \sum_{i,j=1}^{n} \omega_i\omega_j \, \Tilde{L}^{0,k}(f^{(2)})_{ij}.
		\end{align*}
		This implies 
		\begin{align}\label{eq_3.3}
			\sum_{i,j=1}^{n} \omega_i\omega_j \, \Tilde{L}^{0,k}(f^{(2)})_{ij}-c^2\Tilde{L}^{0,k}(f^{(2)})_{00} =-2c 	\Tilde{L}^{1,k}(f^{(2)})_{0}).
		\end{align}
	Next adding \eqref{eq_3.3} and \eqref{eq_3.4} we obtain  $c(1+n) \Tilde{L}^{1,k}(f^{(2)})_{0}=0$.
		This along with \textbf{Step 1} implies $ (f^{(2)})_0=0$. Moreover using \eqref{relation1_1} we conclude that $ (f^{(2)})_i=0$ for $1\le i \le  n$. As a result, we have  $f^{(2)}=0$.
		
		We now drop the assumption $J f^{(2)}=0$. To do so,  we utilize the decomposition given in Lemma~\ref{decompose} and express $f^{(2)}$ as
		\begin{equation*}
			f^{(2)} = A^{(2)} + i_{g_{\frac{1}{c}}} f^{(0)},
			\quad \text{with } JA^{(2)}=0,
		\end{equation*}
		where $f^{(0)}$ and $A^{(2)}$ are explicitly given by
		\begin{equation*}
			f^{(0)} =  (n-{\frac{1}{c^2}})^{-1} Jf^{(2)} \in C_{c}^{\infty}(\mathbb{R}^{1+n}),\quad
		A^{(2)} =f^{(2)}-(n-{\frac{1}{c^2}})^{-1} J f^{(2)}  \in C_{c}^{\infty}(\mathbb{R}^{1+n};S^2).
		\end{equation*}
		Clearly, $\Tilde{L}^{2,k}A^{(2)}=\Tilde{L}^{2,k}f^{(2)}=0$, for $k=0,1,2$ and $JA^{(2)}=0$. This implies that $A^{(2)}=0.$ Therefore, $f^{(2)}=i_{g_{\frac{1}{c}}}f^{(0)}$.\smallskip
		
		\textbf{Step 3.} The case of $m\ge 3$.
		
        We argue by induction on \( m\geq 2 \) to prove the result. We know that Theorem \ref{thm_partsph} holds true for \( m = 2 \), as established in \textbf{Step 2}. Now, let us assume that the above statement holds true for $( m-1 )$. We assume that $\Tilde{L}^{m,k}f^{(m)}((t,x),\omega)=0$ for all $ (t,x) \in \mathbb{R}^{1+n}$, $ \omega \in B_{n}(\omega_0,\delta)$ and $k=0,1,\dots,m$. Now $f^{(m)}$ can be written in the following form
		\[
		f^{(m)}=A^{(m)}+i_{g_{\frac{1}{c}}}f^{(m-2)}, \quad \text{where }JA^{(m
			)}=0.
		\]
		Clearly, $\tilde{L}^{m,k}f^{(m)}=\tilde{L}^{m,k}A^{(m)}$. Therefore, we have 
		\[
		\Tilde{L}^{m,k}A^{(m)}((t,x),\omega)=0, \quad \text{ for all } \omega \in B_{n}(\omega_0,\delta),\quad (t,x) \in \mathbb{R}^{1+n} \quad \text{and } k=0,1,\dots,m.
		\]
		We use tangential gradient over $B_{n}(\omega_0,\delta) $ to get data about the sub-tensor fields of rank $(m-1)$. By considering the \( i \)-th component of the tangential gradient vector, we  obtain the following:
		\[
	\left(	\partial_{\omega_i}L^{m,k}A^{(m)} -[\omega_i \sum_{j=1}^{n}\omega_{j}\partial_{\omega_j}L^{m,k}A^{(m)}]\right)\Big|_{\omega \in B_{n}(\omega_0,\delta) }=0, \quad  \text{for} \, i=0,1,\dots,n.
		\]
		Using Lemma \ref{gen1}, we obtain
		\begin{equation}\label{rankred_m_ten}
			\begin{aligned}
				\Tilde{L}^{m-1,k}(A^{(m)})_i&= \omega_i \sum_{j=1}^{n}\omega_j \Tilde{L}^{m-1,k}(A^{(m)})_j
				=\omega_i [\Tilde{L}^{m,k}A^{(m)}-c \Tilde{L}^{m-1,k}(A^{(m)})_0]\\
				&= -c\omega_i \Tilde{L}^{m-1,k}(A^{(m)})_0, \quad \text{for } k=0,1,\dots, (m-1).
			\end{aligned}
		\end{equation}
		Again, we take tangential gradient over $B_{n}(\omega_0,\delta)$ on the both sides of equation \eqref{rankred_m_ten}. Analyzing the $i$-th component and taking sum over $i$ from $1$ to $n$ we deduce
		\begin{equation}\label{relation2}
			(m-1)\left[\sum_{i=1}^{n}\tilde{L}^{m-2,k}
			(A^{(m)})_{ii}-\sum_{i,j=1}^{n}\omega_i \omega_j\tilde{L}^{m-2,k}(A^{(m)})_{ij}\right]
			=c(1-n)\tilde{L}^{m-1,k}(A^{(m)})_0,
		\end{equation}
		for $k=0,1,\dots,m-1$.
		Now the condition $JA^{(m
			)}=0$ implies \( \Tilde{L}^{m-2,k}(JA^{(m
			)}) = 0 \), for any \( k \).
	Consequently,  $	c^2 \tilde{L}^{m-2,k}(A^{(m)})_{00}=\sum_{i=1}^{n} \tilde{L}^{m-2,k}(A^{(m)})_{ii}, $ for $0\le k \le m-1$.
	This along with   \eqref{relation2} entails
		\begin{equation*}
			2c(m-1)\Tilde{L}^{m-1,k}(A^{(m)})_0 + c(n-1)\Tilde{L}^{m-1,k}(A^{(m)})_0 =0, \text{ for } k=0,1,\dots,m-1.
		\end{equation*}
	Since $n\geq 3$ and $m\geq 2$,  this further implies $ 	\Tilde{L}^{m-1,k}(A^{(m)})_0=0, \text{ for }k=0,1,\dots,m-1.$ 
	Furthermore, combining this equation with \eqref{rankred_m_ten} we conclude
		\begin{equation}
			\Tilde{L}^{m-1,k}(A^{(m)})_i=0, \text{ for }k=0,1,\dots,m-1, \qum{and}\quad  i=0,1,\dots,n.
		\end{equation}
		Next, the condition  $JA^{(m
			)}=0$ implies that \(J(A^{(m)})_i = 0 \), for every index \( i \in \{0, 1, \ldots, n\} \). Thus, by the induction hypothesis, \( (A^{(m)})_i = 0 \), for all \( i = 0, 1, \ldots, n \), which implies that \( A^{(m)} = 0 \), and therefore $f^{(m)}=i_{g_{\frac{1}{c}}}f^{(m-2)}$.
			This completes the proof.
	\end{proof}
	\section{Inversion formula for MLRT of symmetric tensor fields}\label{sec6}
    In the proofs of Theorems  \ref{thm_fullsph} and \ref{thm_partsph}, we discussed the injectivity of symmetric tensor fields and observed that a vector field and the trace part of a symmetric tensor field (of rank $\ge 2$) can  be recovered from the MLRTs.  In this section, we provide an algorithmic way to reconstruct vector fields and trace free higher order symmetric tensor fields from its momentum light ray transforms. 
	\subsection{Reconstruction of vector fields:}
	We now proceed to derive the Fourier slice theorem for LRT [$\tilde{L}^{1,0}$] and MLRT [$\tilde{L}^{1,1}$]. For $f \in C_{c}^{\infty}(\mathbb{R}^{1+n})$, the Fourier transform of $f$ at $\zeta$ is denoted by $\F f(\zeta)=\hat{f}(\zeta)$. The Fourier transform of $f$ on a hyperplane through the origin, say $(1,\om)^{\perp}$, is denoted by $\F_{(1,\om)^{\perp}}f$. To this end, we denote  $\tilde{\omega}=(1,\omega)$. 
	\begin{lemma}{(\textbf{Fourier slice theorem for vector fields}):}
		For any $\om \in \sn$, every $f^{(1)}\in C_{c}^{\infty}(\R^{1+n};S^{1})$, must satisfy
		\begin{eqnarray}
			(1,\omega)\cdot \widehat{f^{(1)}}(\zeta)
			=&\Phi_{1}(\omega,\zeta)
            :=\sqrt{2}\mathcal{F}_{(1,\omega)^{\perp}}(\Tilde{L}^{1,0}f^{(1)}(-,\omega))(\zeta), 
			\quad \mbox{for any } \zeta \perp (1,\omega );  \label{eq_1}\\
			\sum_{j=0}^{n}\Tilde{\omega}_{j}(\tilde{\om}\cdot \nabla_{\zeta})\widehat{f^{(1)}_j}(\zeta)
			=&\Phi_{2}(\omega, \zeta)
			:=2\sqrt{2}\mathcal{F}_{(1,\omega)^{\perp}}(\Tilde{L}^{1,1}f^{(1)}(-,\omega))(\zeta), \quad \mbox{for any } \zeta \perp (1,\omega ).  \label{eq_2}
		\end{eqnarray}
	\end{lemma}
	\begin{proof}
		The identity \eqref{eq_1} is nothing but restricted version of the classical Fourier Slice Theorem for vector fields. The following steps demonstrate the proof of \eqref{eq_1}. For a  fixed $\om \in \S^{n-1}$ 
		\[
		\begin{aligned}
			(1,\om)\cdot \widehat{f^{(1)}}(\zeta)&=\sum_{i_1=0}^{n}\Tilde{\om}_{i_1}\int_{\R^{1+n}}f^{(1)}_{i_1}(t,x)e^{-\i  (t,x)\cdot \zeta }\, dtdx\\
			&=\sqrt{2}\int_{(1,\om)^{\perp}}\Big[\int_{\R}f^{(1)}_{i_1}(l+r(1,\om)dr \Big]e^{-\i l\cdot \zeta}dl\\
			&= \sqrt{2}\mathcal{F}_{(1,\om)^{\perp}}\Big(\tilde{L}^{1,0}f^{(1)}(-,\om)\Big)(\zeta), \quad \text{ for}\, \zeta \perp (1,\om).
		\end{aligned}
		\]
	In order to obtain the relation \eqref{eq_2},  we observe that, for any $\zeta \perp (1,\omega )$
		\[
		\begin{aligned}
			\Phi_{2}(\omega, \zeta) &= 2\sqrt{2}\mathcal{F}_{(1,\omega)^{\perp}}(\Tilde{L}^{1,1}f^{(1)}(-,\omega))(\zeta) \\
			&=2\sqrt{2}\sum_{i_1=0}^n\tilde{\om}_{i_1}\int_{\tilde{\om}^{\perp}}\int_{\R}sf^{(1)}_{i_1}(l+s\tilde{\om})e^{-\i l\cdot \zeta}\,dsdl\\
			&=\sum_{i_1=0}^n\tilde{\om}_{i_1}\int_{\R^{1+n}}(t,x)\cdot \tilde{\om}f^{(1)}_{i_1}(t,x)e^{-\i (t,x)\cdot \zeta}\, dtdx\\
			&=\sum_{i_1=0}^{n}\Tilde{\omega}_{i_1}(\tilde{\om}\cdot \nabla_{\zeta})\widehat{f^{(1)}_j}(\zeta).
		\end{aligned}
		\]
		This completes the proof for \eqref{eq_2}.
	\end{proof}
	
These two equalities [\eqref{eq_1}, \eqref{eq_2}] play an important role in the reconstruction of $f^{(1)}$. We see that, $\widehat{f^{(1)}}$ can be recovered in some open set in $\R^{1+n}$ even if we have partial data of MLRTs of $f^{(1)}$.
Before establishing Theorem \ref{rec_roof_fn}, we first examine the corresponding result for vector fields in $\R^{1+3}$. Subsequently, we will identify the necessary modifications required for vector fields in $\R^{1+n}$. Recall that, $H_n$ is defined as follows 
	\[
	H_n := \{\zeta=(\zeta_0,\zeta') \in \R^{1+n}:|\zeta_0|<|\zeta'|,\quad \zeta \perp (1,\om) \quad \text{for some }\om \in \Bdl \}. 
	\]
	
	Let us present an algorithmic way to calculate $\widehat{f^{(1)}}$ on $H_n$ from $\tilde{L}^{1,k}f^{(1)}((t,x),\omega)$ in order to complete the reconstruction.
	
	\noindent \textbf{Reconstruction Algorithm for $m = 1$ and $n\geq 3$:}
	\begin{enumerate}\setlength\itemsep{3mm}
		\item[Step 1:] Given $\om \in B_n(\om_0,\delta)$, and $\zeta \in H_n$ such that $\zeta \perp (1,\om)$, calculate
		\[  \Phi_1 (\om,\zeta):=  \sqrt{2} \mathcal{F}_{(1,\om)^{\perp}}(\tilde{L}^{1,0}f^{(1)}(-,\om))(\zeta), \qquad
		\Phi_2 (\om,\zeta):= 2\sqrt{2} \mathcal{F}_{(1,\om)^{\perp}}(\tilde{L}^{1,1}f^{(1)}(-,\om))(\zeta),    \]
		where $\mathcal{F}_{(1,\om)^{\perp}}(\cdot)$ denotes the partial Fourier transform on the hyperplane $(1,\om)^{\perp}$.

		\item[Step 2:] Given $\zeta = (\zeta_0,\zeta') \in (1,\om)^{\perp}$, find a rotation matrix $R$ in $\Rn$ such that
		\[
		R\left(\frac{\zeta'}{|\zeta'|}\right) = e_2 \quad \mbox{and}\quad 
		R(\om) = (b_1,b_2,0,\cdots,0)=\w_1.
		\] 
		Construct the $(1+n)\times(1+n)$ matrix 
		$M = \begin{pmatrix}
			1 &0\\ 0 &R
		\end{pmatrix}$.

		\item[Step 3:]  Set $\delta' = \frac{1}{2}\left(\delta-|\om-\om_0|\right)$. For any $\varphi_i \in (-\delta',\delta')$, set $\w_i \in \sn$ for $i=1,\cdots,n$ as in \eqref{phi_i_3} for $n=3$ and \eqref{phi_i} for $n>3$.
		Note that $\w_i$ satisfies $(M\zeta) \perp (1,\w_i)$, and $\{\w_i : i=1,\cdots,n\}$ is linearly independent.

        \item[Step 4:] Use Gram-Schmidt orthogonalization to find a $(n+1)\times(n+1)$ matrix $A$ orthonormalizing 
		$\{(1,\w_i) : i=1,\cdots,n\}$.
		
		\item[Step 5:] Compute the vector $\Phi_3(\zeta)$ and the scalar $\Phi_4(\om,\zeta)$
		\begin{align}\label{eq_3}
	\Phi_3(\zeta) := \sum_{j,k=1}^{n}\Phi_{1}(R^{-1}\w_j,\zeta)[A^TA]_{kj}(1,\w_k)
\quad \mbox{and}\quad 
\Phi_4(\om,\zeta) := \sum_{i,j=1}^n\om_i \om_j\partial_j [M^{-1}\Phi_3(\zeta)]_i.  
		\end{align}

		\item[Step 6:] Finally, calculate
        \quad $\widehat{f^{(1)}}(\zeta)=\frac{1}{2}[\Phi_2-\Phi_4]\zeta - M^{-1}\Phi_3$,  for all $\zeta \in H_n$.
		\end{enumerate}
	
	\begin{proposition}\label{rec_roof_f3}
		Assume that $f^{(1)}$ $\in C_{c}^{\infty}(\mathbb{R}^{1+3};S^{1})$, $\om_0 \in \sn$ and $\delta>0$ small enough. If $\tilde{L}^{1,k}f^{(1)}((t,x),\omega )$ is known for $(t,x)\in \mathbb{R}^{1+3}$, $\omega \in B_{3}(\om_0,\delta)$ and $k=0,1$, then $\widehat{f^{(1)}}(\zeta)$ is known for $\zeta \in H_3$.
	\end{proposition}
	\begin{proof}
		Choose $\zeta=(\zeta_{0},\zeta^{'})\in H_3$, where $\zeta^{'}=(\zeta_{1},\zeta_{2},\zeta_{3})$. 
		That is, $\zeta \perp (1,\omega)$, for some $\omega \in \mathbb{S}^{2}$ satisfying $|\omega  -\omega_{0}|<\delta$ for some fixed $\omega_{0}\in \mathbb{S}^{2}$ and $\delta>0$. 
		Choose $\eta \in \mathbb{R}^{3}$ such that $\eta \perp \{\omega,\zeta^{'}\}$. 
		Take $R$ to be the rotation in $\mathbb{R}^{3}$ which takes $\eta \to e_{3}$ and $\zeta^{'}\to |\zeta^{'}|e_{2}$, where $\{e_1,e_2,e_3\}$ is the canonical basis of $\R^3$. 
		If $R\om = b = (b_1,b_2,b_3)$ then the condition $\eta \perp \omega$ implies that $b_3=0$. Again, the condition $\zeta \perp (1,\omega)$ implies $\zeta_{0}+|\zeta^{'}|b_2=0$.
		For $i=1,2,3$, let us choose 
		\begin{equation}\label{phi_i_3}
		\w_{i}=(b_1\cos\varphi_{i},b_2,b_1\sin\varphi_i), \quad \mbox{where }\varphi_1=0, \quad \varphi_2,\varphi_3 \in (-\delta',\delta').
		\end{equation}
		After the rotation, \eqref{eq_1} can be written as
		\begin{equation}\label{eq3}
			(1,p)\cdot g(\xi)=\Phi_1(R^{-1}p,M^{-1}\xi),\quad \mbox{where } 
			M=\begin{pmatrix}
				1&0\\
				0&R
			\end{pmatrix}, \quad \xi=M\zeta=(\zeta_{0},|\zeta^{'}|e_2),
			\quad 
		\end{equation}
		with $g(\xi) = M\widehat{f^{(1)}}(\zeta)$ and $p\in B_3(b,\delta)$ such that $(1,p) \perp \xi$.

		Notice that, $(1,\w_i)$ is close to $(1,b)=(1,R\omega)$ and $\omega$ is close to $\omega_0$, which implies $(1,\w_i)$ is close to $(1,R\omega_0)$. Therefore, by the choice of $\delta,\delta'>0$, we can choose $p=\w_i$ in the \eqref{eq_1}.
		Denote $\xi^{\perp}$ by $V$. Note that $\{(1,\w_{i}): i=1,2,3\}$ is contained in $V$ using the fact that $\zeta_0+ b_2 |\zeta'|=0$, and they are linearly independent, so they form a basis for $V$. 
		Assume that, $\{\mu_i:i=1,2,3\}$ is any orthonormal basis for $V$. 
		 Let $A=[a_{ij}]_{4\times 4}$ be the basis change matrix fixes $\xi$ and changes the basis of $V$ from $\{(1,\w_{j})\}_{j=1}^3$ to $\{\mu_{i}\}_{i=1}^3$. That is, we have
        \begin{equation}\label{basis_ch}
        A\xi = \xi, \quad \mu_i = \sum_{j=1}^{3}a_{ij}(1,\w_j), \quad \mbox{for }i=1,2,3.
        \end{equation}
         Hence, $g(\xi)$ can be re-written as 
		\begin{equation*}
			g(\xi)=c_{0}\xi +\sum_{i=1}^{3}c_{i}\mu_{i}, \quad \mbox{where}\quad c_i=g(\xi)\cdot \mu_{i}.
		\end{equation*}
		So, \eqref{eq3} implies
		\[
	c_i= \sum_{j=1}^3 a_{ij} (g(\xi)\cdot (1,\w_j))= \sum_{j=1}^{3} a_{ij}\Phi_{1}(R^{-1}\w_{j},M^{-1}\xi).
		\]
		Using this form of $c_i$ in the expression of $g(\xi)$, we get
		\[\begin{aligned}
			M\widehat{f^{(1)}}(\zeta)=&g(\xi)=c_0 \xi +\sum_{j,k}\Phi_{1}(R^{-1}\w_{j},M^{-1}\xi)[A^{T}A]_{kj}(1,\w_k)\\ 
			\implies \quad
			\widehat{f^{(1)}}(\zeta)=&c_0\zeta+M^{-1}\Phi_3,
			\quad\mbox{where}\quad \Phi_3=\sum_{j,k}\Phi_{1}(R^{-1}\w_{j},\zeta)[A^{T}A]_{kj}(1,\w_k).
		\end{aligned}\]
		Now substituting form of $\widehat{f^{(1)}}$ in \eqref{eq_2}, we see
		\[
		c_0(\omega,\zeta)=\frac{1}{2}[\Phi_2(\omega,\zeta)-\Phi_4(\omega,\zeta)],
		\quad \mbox{where} \quad \Phi_4=\sum_{p,i}\Tilde{\omega}_{p}\Tilde{\omega}_{i}\partial_{p}[M^{-1}\Phi_3]_i.
		\]
		Therefore, varying $\zeta \in H_3$ we have 
		\[
		\widehat{f^{(1)}}(\zeta)=\frac{1}{2}[\Phi_2-\Phi_4]\zeta - M^{-1}\Phi_3, \quad \forall \zeta \in H_3.
		\]
		This completes the proof.
	\end{proof}
	The complete recovery of the vector field \( f^{(1)} \) proceeds through the following sequence of steps. Since the components of the vector field $f^{(1)}$ are compactly supported, the Paley-Wiener theorem implies that the components of \( \widehat{f^{(1)}} \) are analytic. Proposition \ref{rec_roof_f3} ensures that \( \widehat{f^{(1)}} \) is known in the open set \( H_3 \) (contains space-like vectors). By analytic continuation, this determines \( \widehat{f^{(1)}} \) uniquely on the whole of \( \mathbb{R}^{1+3} \). Applying the inverse Fourier transform then yields the reconstruction of \( f^{(1)} \).
	\begin{remark}
		We want to highlight the fact that, in the proof, there might be several choices for $\{\mu_{i}:i=0,1,2,3\}$ (orthonormal basis for $V=\xi^{\perp}$). For instance, Gram-Schmidt orthonormalization of  $\{(1,\omega_i):i=0,1,2\}$ provides us with a choice. 
	\end{remark}
Now we see the reconstruction of vector fields defined on $\mathbb{R}^{1+n}$.
\begin{proof}[\textbf{Proof of Theorem \ref{rec_roof_fn}
}]
For $n=3$, the proof is presented in Proposition \ref{rec_roof_f3}. Here, let us assume $n\geq 4$.
Let $\zeta  =(\zeta_0,\zeta^{'}) \in H_n $ be arbitrary. Assume that $\zeta \perp (1,\omega)$ for some  $\omega \in  B_{n}(\omega_0,\delta)$. Also consider $\{\eta_{i}:i=1,2,\dots, (n-2)\}$, set of vectors in $\R^{n}$ satisfying the following
		\[
		\eta_{i}\perp \{\om, \zeta^{'} \}\, \quad \text{and} \quad \eta_i\perp \eta_j \quad \text{for}\quad i\neq j, \quad \text{ where}\quad  i,j=1,2,\dots ,(n-2). 
		\]
		Let $R$ be a rotation matrix in $\mathbb{R}^{n}$, which satisfies the following
		\[
		\left\{\begin{aligned}
			R\eta_i=&e_{i+2}\quad \text{ for }i=1,2,\dots, (n-2),\\
			R\zeta^{'}=&|\zeta^{'}|e_{2}.
		\end{aligned}
		\right.
		\]
		Assume that, $\omega$ goes to $b=(b_1,\dots,b_n)$ under the rotation $R$. Notice that,
		\[\left\{
		\begin{aligned}
			\eta_i\perp &\omega \quad \text{implies that}\quad b_i=0 \text{ for} \, i=3,\dots,(n-2),\\
			\zeta\perp &(1,\omega) \quad \text{implies that}\quad \zeta_0+b_2|\zeta^{'}|=0.
		\end{aligned}
		\right.
		\]
		Consider $M$= 
		$\begin{pmatrix}
			1&0\\
			0&R
		\end{pmatrix}$. Now \eqref{eq_1} can be rewritten as
		\begin{equation}\label{reinv1}
			(1,p)\cdot g(\xi)=\Phi_1(R^{-1}p,M^{-1}\xi),\quad \text{where}\quad \xi=M\zeta=(\zeta_{0},|\zeta^{'}|e_2),\, 
		\end{equation}
		with $g(\xi)=M\widehat{f^{(1)}}(\zeta)$ and $p\in B_n(b,\delta)$ satisfying $\xi \perp (1,p)$. Let us choose 
\begin{equation}\label{phi_i}
\begin{aligned}
\w_{1}=&(b_1,b_2,0,\cdots,0)\\
\w_{2}=&(b_1\cos\varphi_{1},\,b_2,\,b_1\sin\varphi_{1},0,\cdots,0)\\
\w_{3}=& (b_1\cos\varphi_{1},\,b_2,\,b_1\sin\varphi_{1}\cos \varphi_{2},\, b_1\sin\varphi_{1}\sin \varphi_{2},0,\cdots,0)\\
&\vdots\\
\w_{n-1}=& (b_1\cos\varphi_{1},\,b_2,\,b_1\sin\varphi_{1}\cos \varphi_{2},\cdots,\, b_1\sin\varphi_{1}\cdots\sin \varphi_{n-3}\cos \varphi_{n-2},\, b_1\sin\varphi_{1}\cdots\sin \varphi_{n-3}\sin \varphi_{n-2})\\
\w_{n}=& (b_1\cos\varphi_{1},\,b_2,\,b_1\sin\varphi_{1}\cos \varphi_{2},\cdots,\, b_1\sin\varphi_{1}\cdots\sin \varphi_{n-3}\cos \varphi_{n-1},\, b_1\sin\varphi_{1}\cdots\sin \varphi_{n-3}\sin \varphi_{n-1}),
\end{aligned}
\end{equation}
where $\varphi_{i} \neq 0$ for $1\leq i \leq n-1$ and $\varphi_{n-1} \neq \varphi_{n-2}$.
Furthermore, we choose $\varphi_{i}$'s are  small enough to guarantee that $p$ can be replaced by $\w_i$. Notice that, $\{(1,\w_i):i=1,2,\dots,n\}$ is a linearly independent set in $\xi^{\perp}=V$. Therefore, $\{(1,\w_i):i=1,2,\dots,n\}$ is a basis for $V$. Assume that, $\{\mu_{i}:i=1,2,\dots,n\}$ is an orthonormal basis for $V$ and $A=[a_{ij}]_{(n+1)\times (n+1)}$ be the change of basis matrix from $\{\xi,(1,\w_{i})\}$ to $\{\xi,\mu_{i}\}$, as in \eqref{basis_ch}.
        Now, following a similar process as for $n=3$, we get 
		\[
		g(\xi)=c_0 \xi +\sum_{j,k=1}^{n}\Phi_{1}(R^{-1}\w_{j},M^{-1}\xi)[A^{T}A]_{kj}(1,\w_k).
		\]
		The above implies 
		\[
		\widehat{f^{(1)}}(\zeta)=c_0\zeta+M^{-1}\Phi_3,
		\]
		where $\Phi_3=\sum_{j,k=1}^{n}\Phi_{1}(R^{-1}\w_{j},\zeta)[A^{T}A]_{kj}(1,\w_k)$. Now substituting this form of $\widehat{f^{(1)}}$ in \eqref{eq_2}, we obtain
		\[
		c_0(\omega,\zeta)=\frac{1}{2}[\Phi_2(\omega,\zeta)-\Phi_4(\omega,\zeta)],
		\] where $\Phi_4=\sum_{p,i=1}^{n}\Tilde{\omega}_{p}\Tilde{\omega}_{i}\partial_{p}[M^{-1}\Phi_3]_i.$
		Therefore, we have 
		\[
		\widehat{f^{(1)}}(\zeta)=\frac{1}{2}[\Phi_2-\Phi_4]\zeta - M^{-1}\Phi_3, \qquad \mbox{for}\quad \zeta \in H_n.
		\]
	\end{proof}
	\begin{remark}
		In the above proof, we can choose $\{\mu_i\}$ as orthonormalized vectors of $\{(1,\w_i)\}$ following Gram-Schmidt process. But calculating the change of basis matrix is a little bit lengthier in this case.
	\end{remark}
	
	\subsection{Reconstruction of symmetric $m$-tensor fields:}
	At the heart of this sub-section lies proof of the Theorem \ref{rec_mt}. We approach the proof using the method of induction on rank of the tensor field, $m \,(\geq 2)$. Trace free decomposition of a symmetric 2-tensor field $f^{(2)}$ allows us to write 
	\[
	f^{(2)}=A^{(2)}+ i_gf^{(0)},
	\]
	where $g=\diag(-1,1,\dots,1)$;  $A^{(2)}$ is symmetric 2-tensor satisfying  $j_{g}A^{(2)}= -A^{(2)}_{00}+ \sum_{p=1}^{n} A^{(2)}_{pp}=0$ and $f^{(0)}$ is a complex valued function. We have already seen, $\tilde{L}^{2,k}(i_g f^{(0)})=0$.
	So, we cannot recover the trace part using MLRTs. Recovery can be done for trace-free part only. We use the tangential derivative to get MLRTs of a 1-tensor field [ restricting one index in the 2-tensor field ].
	\begin{proposition}
		\label{rec_2t}
		For $n\geq3$, assume that $f^{(2)}\in C_{c}^{\infty}(\mathbb{R}^{1+n};S^2)$ satisfies $j_gf^{(2)}=0$. For a fixed $\omega_0 \in \mathbb{S}^{n-1}$ and $\delta >0$, if $\tilde{L}^{2,k}f^{(2)}$ $((t,x),\omega)$ is known for $k=0,1,2$, $(t,x)\in \mathbb{R}^{1+n}$, $\omega \in B_{n}(\omega_0,\delta)$  then we can recover $\widehat{f^{(2)}}$ in the open set $H_n$.
	\end{proposition}
	\begin{proof}
		Let's calculate tangential gradient of $\tilde{L}^{2,0}f^{(2)}(-,\om)$ on $B_{n}(\omega_0,\delta)$. The $i$-th component of the tangential gradient of $\Tilde{L}^{2,0}f^{(2)}$ can be written as 
		\begin{equation*}
			(\nabla_{S}\tilde{L}^{2,0}f^{(2)})_{i}
			=2\tilde{L}^{1,0}(f^{(2)})_{i}+\partial_{x_i}\tilde{L}^{2,1}f^{(2)}-\omega_{i}\sum_{j=1}^{n}\omega_{j}[2\tilde{L}^{1,0}(f^{(2)})_{j}+\partial_{x_j}\tilde{L}^{2,1}f^{(2)}], \text{ [see Lemma \ref{lemma2.7}]}.
		\end{equation*}
	We rewrite the above equation in the following way
		\begin{equation}\label{rec_1}
			\tilde{L}^{1,0}(f^{(2)})_{i}=-\omega_{i}\tilde{L}^{1,0}(f^{(2)})_{0}-\Psi_{1}^{(0),(i)},
		\end{equation}
		where $\Psi_{1}^{(0),(i)}:=\frac{1}{2}\Big[\partial_{x_i}\tilde{L}^{2,1}f^{(2)}-\omega_{i}\sum_{j}\omega_{j}\partial_{x_j}\tilde{L}^{2,1}f^{(2)}-(\nabla_{S}\tilde{L}^{2,0}f^{(2)})_{i}\Big]+\omega_i \tilde{L}^{2,0}f^{(2)}$, contains known terms only. In the notation \( \Psi_{1}^{(0),(i)} \), the superscript \((0)\) indicates that we start with zero moment, while \((i)\) specifies that we are considering the $i$-th component of the tangential gradient.
		Again take tangential gradient over $B_{n}(\omega_0,\delta)$ in both sides of \eqref{rec_1}. Considering the $i$-th component, we get the following
		\begin{equation*}
			\begin{aligned}
				&\tilde{L}^{0,0}(f^{(2)})_{ii}
				- \omega_{i} \sum_{j} \omega_{j} \tilde{L}^{0,0}(f^{(2)})_{ij}
				+ \partial_{x_i} \tilde{L}^{1,1}(f^{(2)})_{i}
				- \omega_{i} \sum_{j} \omega_{j} \partial_{x_j} \tilde{L}^{1,1}(f^{(2)})_{i} \\
				&\quad =
				- \tilde{L}^{1,0}(f^{(2)})_{0}
				- \omega_{i} \tilde{L}^{0,0}(f^{(2)})_{i0}
				+ \omega_{i}^{2} \sum_{j} \omega_{j} \tilde{L}^{0,0}(f^{(2)})_{0j} \\
				&\qquad
				- \omega_{i} \partial_{x_i} \tilde{L}^{1,1}(f^{(2)})_{0}
				+ \omega_{i}^{2} \sum_{j} \omega_{j} \partial_{x_j} \tilde{L}^{1,1}(f^{(2)})_{0}
				+ \omega_{i}^{2} \tilde{L}^{1,0}(f^{(2)})_{0}
				- (\nabla_{S} \Psi_{1}^{(0),(i)})_{i}.
			\end{aligned}
		\end{equation*}
		Note that, in the above deduction, we have used the result $\partial_{\omega_{i}}L^{1,0}f^{(2)}=L^{0,0}(f^{(2)})_{0}+\partial_{x_i}L^{1,1}f^{(2)}$.
		Taking sum over $i$ on both sides in the above equation, we get
		\begin{equation*}
			\begin{aligned}
			(1-n)\tilde{L}^{1,0}(f^{(2)})_{0}+\sum_{i}[\nabla_{S}(\Psi_{1}^{(0),(i)})]_{i}=&	\sum_{i}\tilde{L}^{0,0}(f^{(2)})_{ii}-\sum_{i,j}\omega_{i}\omega_{j}\tilde{L}^{0,0}(f^{(2)})_{ij}\\&+\sum_{i}\partial_{x_i}\tilde{L}^{1,1}(f^{(2)})_{i}-\sum_{i,j}\omega_{i}\omega_{j}\partial_{x_j}\tilde{L}^{1,1}(f^{(2)})_{i}.
			\end{aligned}
		\end{equation*}
		Now we shall use \eqref{rec_1} in the above equation to get
		\begin{equation*}
			\begin{aligned}
		&	(1-n)\tilde{L}^{1,0}(f^{(2)})_{0}+\sum_{i}[\nabla_{S}(\Psi_{1}^{(0),(i)})]_{i}\\=	&\sum_{i}\tilde{L}^{0,0}(f^{(2)})_{ii}-\sum_{i,j}\omega_{i}\omega_{j}\tilde{L}^{0,0}(f^{(2)})_{ij}+\sum_{i}\partial_{x_i}[-\omega_{i}\tilde{L}^{1,1}(f^{(2)})_{0}+\Psi_{1}^{(1),(i)}]\\&\quad-\sum_{i,j}\omega_{i}\omega_{j}\partial_{x_j}[-\omega_{i}\tilde{L}^{1,1}(f^{(2)})_{0}+\Psi_{1}^{(1),(i)}].
			\end{aligned}
		\end{equation*}
		That is
		\begin{equation*}
			\begin{aligned}
				\sum_{i}\tilde{L}^{0,0}(f^{(2)})_{ii}-\sum_{i,j}\omega_{i}\omega_{j}\tilde{L}^{0,0}(f^{(2)})_{ij}+(n-1)\tilde{L}^{1,0}(f^{(2)})_{0}\\=\sum_{i}[\nabla_{S}(\Psi_{1}^{(0),(i)})]_{i}-\sum_{i}\partial_{x_i}\Psi_{1}^{(1),(i)}+&\sum_{i,j}\omega_{i}\omega_{j}\partial_{x_j}\Psi_{1}^{(1),(i)}.
			\end{aligned}
		\end{equation*}
		The condition $j_gf^{(2)}=0$ gives us $(f^{(2)})_{00}=\sum_{p=1}^n(f^{(2)})_{pp}$. Notice that, 
        \[ \sum_{i,j=1}^n\omega_{i}\omega_{j}\tilde{L}^{0,0}(f^{(2)})_{ij}=\tilde{L}^{2,0}f^{(2)}-2\tilde{L}^{1,0}(f^{(2)})_{0}+\tilde{L}^{0,0}(f^{(2)})_{00}.\] 
        Using these two equalities, from the above equation we get 
		\begin{equation*}
			\begin{aligned}
				2\tilde{L}^{1,0}(f^{(2)})_{0}-\tilde{L}^{2,0}f^{(2)}+(n-1)\tilde{L}^{1,0}(f^{(2)})_{0}=&\sum_{i}[\nabla_{S}(\Psi_{1}^{(0),(i)})]_{i}-\sum_{i}\partial_{x_i}\Psi_{1}^{(1),(i)}+\sum_{i,j}\omega_{i}\omega_{j}\partial_{x_j}\Psi_{1}^{(1),(i)}\\
				\implies \tilde{L}^{1,0}(f^{(2)})_{0}=&\Psi_2^{(0)},
			\end{aligned}
		\end{equation*}
		where $\Psi_2^{(0)}=\frac{1}
		{(n+1)}[\sum_{i}[\nabla_{S}(\Psi_{1}^{(0),(i)})]_{i}-\sum_{i}\partial_{x_i}\Psi_{1}^{(1),(i)}+\sum_{i,j}\omega_{i}\omega_{j}\partial_{x_j}\Psi_{1}^{(1),(i)}+\tilde{L}^{2,0}f^{(2)}]$, contains known terms only. In the notation $\Psi_2^{(0)}$, the superscript $(0)$ denotes that we are working with zero moment ( $\tilde{L}^{2,0}\,)$.

		So, the knowledge of LRT of symmetric $2$-tensors is now reduced to the LRT of a vector field. 
		Similarly, if we proceed with the 1st moment, $\tilde{L}^{2,1}$, then we end up with MLRT of the $0$-th column of $f^{(2)}$.
		The relation \eqref{rec_1} helps to get $\Tilde{L}^{1,k}(f^{(2)})_{i}$ for $k=0,1$ and for any $i=1,2,\dots,n$.
		Therefore, using the reconstruction of vector fields Theorem \ref{rec_roof_fn}, we get the Fourier transform of these vector fields, $(f^{(2)})_i$ for $i=0,1,\dots,n$ in some open set in $\mathbb{R}^{1+n}$.
		As we have started with compactly supported components of $f^{(2)}$, $\widehat{f^{(2)}}$ can be extended to $\R^{1+n}$ uniquely, using analytic continuation. 
		Therefore, we can recover $f^{(2)}$ in $\mathbb{R}^{1+n}$.
	\end{proof}
	\begin{proof}[\textbf{Proof of Theorem \ref{rec_mt}}]
		We use the method of induction on $m\geq 2$, the rank of symmetric tensor fields. For $m=2$, the statement holds true [ due to Proposition \ref{rec_2t} ]. Let us assume that, the statement holds true for all symmetric tensor fields of rank $(m-1).$ 
		
		We start with $f^{(m)}\in C_{c}^{\infty}(\R^{1+n};S^m)$ for which $\tilde{L}^{m,k}f^{(m)}((t,x),\om)$ is known, $(t,x)\in \R^{1+n}$, $\om \in \Bdl$ and $k=0,\dots ,m.$ Let us take tangential gradient on $\tilde{L}^{m,k}f^{(m)}$ over the part of the sphere, $\Bdl$. We follow the exact same procedure, what we have done for $m=2$ in Proposition \ref{rec_2t}. Now if we consider the $i$-th component, then we obtain 
		\begin{equation}\label{rec_1_m}
			\tilde{L}^{m-1,k}(f^{(m)})_i=-\om_i\tilde{L}^{m-1,k}(f^{(m)})_0+\Psi_1^{(m),(k),(i)}, \text{  for } k=0,\dots,(m-1),
		\end{equation}
		where 
		\[
		\Psi_1^{(m),(k),(i)}=\frac{1}{m}\Big[\partial_{x_i}\tilde{L}^{m,k+1}f^{(m)}-\omega_{i}\sum_{j}\omega_{j}\partial_{x_j}\tilde{L}^{m,k+1}f^{(m)}-(\nabla_{S}\tilde{L}^{m,k}f^{(m)})_{i}\Big]+\omega_i \tilde{L}^{m,k}f^{(m)}.
		\]
		Again taking tangential gradient over $\Bdl$ on both sides of \eqref{rec_1_m} and summing over $i$, we conclude
		\[
		\begin{aligned}
			(m-1)\Big[\sum_{i=1}^{n}\tilde{L}^{m-2,k}(f^{(m)})_{ii}-\sum_{i,j=1}^{n}\om_i \om_j\tilde{L}^{m-2,k}(f^{(m)})_{ij}\Big] \\+\Big[\sum_{i}\partial_{x_i}\Psi_{1}^{(m),(k),(i)}-\sum_{i,j}\om_i\om_j\partial_{x_j}\Psi_{1}^{(m),(k),(i)}\Big] \\
			=(1-n)\tilde{L}^{m-1,k}(f^{(m)})_0+\sum_{i}[\nabla_{S}\Psi_{1}^{(m),(k),(i)}]_{i}.
		\end{aligned}
		\]
		Now using the condition  $Jf^{(m)}=0$, we obtain
		\[
		\begin{aligned}
			\tilde{L}^{m-1,k}(f^{(m)})_0 &=\frac{1}{2m+n-3}\Big[\sum_{i}[\nabla_{S}\Psi_{1}^{(m),(k),(i)}]_i-\sum_{i}\partial_{x_i}\Psi_{1}^{(m),(k),(i)} \\&\quad+\sum_{i,j}\om_i \om_j\partial_{x_j}\Psi_{1}^{(m),(k),(i)}\Big]=\Psi_2^{(m),(k),(i)}\text{ for } k=0,\dots,m-1.
		\end{aligned}
		\]
		Also using equation \eqref{rec_1_m}, we get $\tilde{L}^{m-1,k}(f^{(m)})_i$ for $i=1,\dots,n$. Now using induction hypothesis, we can recover $\widehat{(f^{(m)})_{i}}$ in an open set $H_n$ for $i=0,1,\dots,n.$ Therefore, we have $\widehat{f^{(m)}}$ in the open set $H_n$.
	\end{proof}
	\section*{Acknowledgments}
	    TM acknowledges UGC for the PhD research fellowship. SKS is supported by IIT Bombay seed grant (RD/0524-IRCCSH0-021) and ANRF Early Career Research  Grant (ECRG) (RD/0125-ANRF000-016). SB thanks the Department of Mathematics, IISER Bhopal for support and encouragement.


\end{document}